\documentclass[11pt,dvipsnames]{amsart}
\usepackage{amsmath,amssymb,amsfonts}
\usepackage[mathscr]{eucal}

\usepackage[left=1.2in,top=2cm,right=1.2in,bottom=2cm]{geometry}
\usepackage[toc,page]{appendix}

\usepackage{tikz}%
\usetikzlibrary{matrix,arrows,decorations.pathmorphing,cd,patterns,calc,backgrounds,patterns}
\tikzset{dummy/.style= {circle,fill,draw,inner sep=0pt,minimum size=1.2mm}}
\tikzset{vertex/.style={fill, circle, minimum size=.1cm, inner sep=0pt}}
\usepackage{tqft}

\makeatletter
\pgfqkeys{/tikz/commutative diagrams}{
  row sep/.code={\tikzcd@sep{row}{#1}{}},
  column sep/.code={\tikzcd@sep{column}{#1}{}},
  bo row sep/.code={\tikzcd@sep{row}{#1}{between origins}},
  bo column sep/.code={\tikzcd@sep{column}{#1}{between origins}},
  bo column sep/.default=scriptsize,
  bo row sep/.default=scriptsize,
}
\def\tikzcd@sep#1#2#3{
  \pgfkeysifdefined{/tikz/commutative diagrams/#1 sep/#2}%
    {\pgfkeysalso{/tikz/#1 sep={\ifx\\#3\\1*\else1.7*\fi\pgfkeysvalueof{/tikz/commutative diagrams/#1 sep/#2},#3}}}%
    {\pgfkeysalso{/tikz/#1 sep={#2,#3}}}}
\makeatother

\usepackage{multirow}
\usepackage{stmaryrd}
\usepackage{comment}

\usepackage[
backend=biber,
style=alphabetic,
maxbibnames=99,
]{biblatex}
\bibliography{references.bib}

\usepackage[colorlinks, citecolor=PineGreen, linkcolor=PineGreen,urlcolor=gray]{hyperref}
\hypersetup{linktoc=all,} 

\usepackage[final]{microtype}

\usepackage{enumerate}

\numberwithin{equation}{section} 
\numberwithin{figure}{section}
\usepackage[nameinlink,capitalise,noabbrev]{cleveref}

\newcommand{\newrefformat}[2]{}
\crefname{lemma}{Lemma}{Lemmas}
\crefname{theorem}{Theorem}{Theorems}
\crefname{definition}{Definition}{Definitions}
\crefname{proposition}{Proposition}{Propositions}
\crefname{remark}{Remark}{Remarks}
\crefname{corollary}{Corollary}{Corollaries}
\crefname{equation}{Equation}{Equations}
\crefname{construction}{Construction}{Constructions}
\crefname{ex}{Example}{Examples}
\crefname{appsec}{Appendix}{Appendices}
\crefname{subsection}{Subsection}{Subsections}

\theoremstyle{plain}
\newtheorem{theorem}[equation]{Theorem}
\newtheorem{corollary}[equation]{Corollary}
\newtheorem{proposition}[equation]{Proposition}
\newtheorem{lemma}[equation]{Lemma}

\newtheorem{introtheorem}{Theorem}
 
\crefname{introtheorem}{Theorem}{Theorems}

\theoremstyle{definition}
\newtheorem{definition}[equation]{Definition}
\newtheorem{example}[equation]{Example}
\newtheorem{remark}[equation]{Remark}

\newtheorem*{notation}{Notation}

\newcommand{\RR}{\mathbb{R}}

\newcommand{\ZZ}{\mathbb{Z}}
\newcommand{\CC}{\mathbb{C}}

\newcommand{\abs}[1]{\lvert#1\rvert}

\newcommand{\dcat}[1]{\mathbb{ #1}}
\newcommand{\cat}[1]{\mathscr{#1}}

\newcommand{\Cat}{\mathrm{Cat}}
\newcommand{\Seg}{\mathrm{Seg}}

\newcommand{\Fin}{\mathrm{Fin}}

\newcommand{\Set}{\mathrm{Set}}

\newcommand{\Fun}{\operatorname{Fun}}

\renewcommand{\hom}{\operatorname{Hom}}

\newcommand{\Sq}{\operatorname{Sq}}

\newcommand{\id}{\operatorname{id}}

\newcommand{\ob}{\operatorname{ob}}
\newcommand{\op}{\operatorname{op}}

\newcommand{\quot}{\twoheadrightarrow}
\newcommand{\cof}{\rightarrowtail}

\renewcommand{\phi}{\varphi}
\newcommand{\bndry}{\partial}

\newcommand{\tinysquare}
{~\begin{tikzpicture}[scale=0.35]
    \fill (0,0) circle (3pt);
    \fill (1,0) circle (3pt);
    \fill (1,1) circle (3pt);
    \fill (0,1) circle (3pt);
    
    \draw (0.2,1) --(0.8,1);
    \draw (0.2,0) --(0.8,0);
    \draw (0,0.2) --(0,0.8);
    \draw (1,0.2) --(1,0.8);

    \draw (0.35,0.35) --(0.35,0.65);
    \draw (0.35,0.35) --(0.65,0.35);
    \draw (0.35,0.65) --(0.65,0.65);
    \draw (0.65,0.35) --(0.65,0.65);
\end{tikzpicture}~}

\newcommand{\tinyspan}{~\begin{tikzpicture}[scale=0.35]
    \fill (0,0) circle (3pt);
    \fill (1,1) circle (3pt);
    \fill (0,1) circle (3pt);
    
    \draw (0.2,1) --(0.8,1);
    \draw (0,0.2) --(0,0.8);
\end{tikzpicture}~}

\newcommand{\tinycospan}{~\begin{tikzpicture}[scale=0.35]
    \fill (1,0) circle (3pt);
    \fill (0,0) circle (3pt);
    \fill (1,1) circle (3pt);
    
    \draw (0.2,0) --(0.8,0);
    \draw (1,0.2) --(1,0.8);
\end{tikzpicture}~}

\input xy
\xyoption{all}

\usepackage[
textwidth=3cm,
textsize=tiny,
colorinlistoftodos]
{todonotes} 

\newcommand{\mcnote}[1]{\todo[color=red!40,linecolor=red!40!black,size=\tiny]{MC: #1}}
\newcommand{\mcnoteil}[1]{\ \todo[inline,color=red!40,linecolor=red!40!black,size=\normalsize]{MC: #1}}

\title{Squares $K$-theory and 2-Segal spaces}

\author[M.E.\ Calle and M.\ Sarazola]{Maxine E.\ Calle and Maru Sarazola}

\address{Department of Mathematics, University of Pennsylvania, Philadelphia, PA 19104}
\email{callem@sas.upenn.edu}

\address{School of Mathematics, University of Minnesota, Minneapolis MN, 55455}
\email{maru@umn.edu}

\date{\today}

\keywords{2-Segal spaces, squares K-theory, Waldhausen $S_\bullet$-construction, double categories}

\subjclass[2020]{19D99, 
18N10, 
55U10, 
18G90. 
}
\begin{document}

\begin{abstract}
We define an $S_\bullet$-construction for squares categories, and introduce a class of squares categories we call \textit{proto-Waldhausen} which capture the properties required for the $S_\bullet$-construction to model the $K$-theory space. 
The primary question we investigate is when the $S_\bullet$-construction of a squares category produces a $2$-Segal space. We show that the answer to this question is affirmative when the squares category satisfies certain ``stability'' conditions. 
\end{abstract}

\maketitle

\tableofcontents

\section{Introduction}

Given a category $\cat C$, its nerve $N_*\cat C$ is a simplicial set whose $n$-simplices are length-$n$ sequences of composable morphisms. The nerve has the extra structure of a \textit{$1$-Segal set}, meaning that the \textit{Segal maps}\[
N_n\cat C\rightarrow N_1\cat C\times_{N_0\cat C}\dots \times_{N_0\cat C} N_1\cat C
\] which send an $n$-simplex to its spine are isomorphisms for $n\geq 2$. In fact, this property completely characterizes nerves: a simplicial set is $1$-Segal if and only if it is isomorphic to the nerve of some category. In particular, any $1$-Segal set $X$ has a composition map given by\[
X_1\times_{X_0} X_1\xleftarrow{\cong} X_2 \xrightarrow{d_1} X_1
\] and relationships between higher simplicies encode the fact that this composition is associative.

One can consider a higher-dimensional version of the Segal condition on a simplicial set $X$, by instead requiring that the maps
\[
X_n\rightarrow X_2\times_{X_1}\dots \times_{X_1} X_2
\]
be isomorphisms for all $n\geq 3$. This gives rise to the notion of \textit{$2$-Segal sets} (also known as \textit{decomposition spaces}), introduced by Dyckerhoff and Kapranov \cite{dyckerhoff/kapranov:19} and G\'alvez-Carrillo, Kock, and Tonks \cite{GCKT:18}. A $2$-Segal set may not have a well-defined composition, but in some sense remembers the higher-dimensional data of associativity. This structure can be thought of as an associative multi-valued composition (as explained in \cite{BOORS:2016}) or as an $A_\infty$-algebra structure in a category of spans of sets \cite{Stern:19}. In fact, this notion can be defined more generally in the setting of simplicial \emph{objects} in any category $\mathcal{C}$. When $\mathcal{C}$ is the category of topological spaces, the natural requirement is for the Segal maps to be weak homotopy equivalences, and we obtain the notion of \emph{2-Segal space}.

Just like all $1$-Segal sets can be obtained as nerves, $2$-Segal sets are the natural output of the $S_\bullet$-construction, which was introduced by Waldhausen \cite{waldhausen:1983} in the context of higher algebraic $K$-theory. Indeed, in \cite{BOORS:2016}, Bergner, Osorno, Ozornova, Rovelli, and Scheimbauer show that every $2$-Segal set arises from a version of the $S_\bullet$-construction, by replacing the classical $K$-theoretical inputs with a different categorical structure called \textit{stable augmented double categories}. A similar result also holds for the case of 2-Segal spaces, where it was long understood that the $S_\bullet$-construction of certain kinds of input categories, such as the \textit{proto-exact categories} of \cite{dyckerhoff/kapranov:19}, has the structure of a $2$-Segal space. Although not every Waldhausen category produces a $2$-Segal space, every $2$-Segal space can be obtained as the $S_\bullet$-construction of a topological generalization of stable augmented double categories \cite{BOORS:2021}.

In this paper, we continue the study of the strong connection between $2$-Segal objects and higher algebraic $K$-theory. The stable augmented double categories of \cite{BOORS:2016} are examples of  \textit{squares categories}, a framework introduced in \cite{CKMZ:squares}
 to carry out new investigations
related to the $K$-theory of varieties \cite{campbell, camzak:24} and scissors congruence problems
\cite{zakharevich:12, malkiewich, hoekzema/merling/murray/rovi/semikina:2021, camzak:22}. The central question of this paper is: when does the $K$-theory construction of a squares category produce a Segal object? 

The $K$-theory of a squares category is given by a double-nerve construction, also called the $T_\bullet$-construction. The name is inspired by the \textit{Thomason construction} for Waldhausen categories (see \cite[Section 1.3]{waldhausen:1983}) which provides another model for the $K$-theory of a Waldhausen category equivalent to the one constructed via $S_\bullet$. Although $T_\bullet \dcat C$ describes a $1$-Segal object in $\Cat$ (see \cref{rmk:Tdot1segalincat}), in most cases it does not produce a $2$-Segal space.

In this paper, we introduce a $S_\bullet$-construction for squares categories, denoted $S^{\square}_\bullet$, which can produce $2$-Segal spaces. Our first result is give sufficient conditions for the $S^{\square}_\bullet$-construction to coincide with the $T_\bullet$-construction.

\begin{introtheorem}[\cref{thm:square and S comparison} and \cref{rmk:WITissquares}]
    There is a functor\[
    S^{\square}_\bullet \colon 
    \mathrm{SqCat}\rightarrow\mathrm{sCat}
    \] from the category of squares categories to the category of simplicial categories such that:\begin{itemize}
        \item[(a)] When the input is a proto-Waldhausen category in the sense of \cref{defn:protoWald cat}, then there is an equivalence of spaces \[
        \Omega\abs{S^{\square}_\bullet\dcat C} \simeq K^{\square}(\dcat C)
        \] with the square $K$-theory defined in \cite{CKMZ:squares}.
        \item[(b)] When the input is a pointed stable double category in the sense of \cite{BOORS:2016}, the composition of functors\[
     \mathrm{SqCat}\xrightarrow{S^{\square}_\bullet} \mathrm{sCat} \xrightarrow{\ob} \mathrm{sSet}
        \]is the $S_\bullet$-construction of \cite{BOORS:2016}.
    \end{itemize}
\end{introtheorem}
This $S^{\square}_\bullet$-construction is also closely related to that of another double-categorical structure for $K$-theory: ECGW-categories. Every ECGW-category $\dcat C$ can be considered as a squares category $U\dcat C$, and the simplicial category $S^{\square}_\bullet U\dcat C$ defined in this paper is the underlying vertical category of the simplicial double category $S_\bullet \dcat C$ of \cite{SS-CGW}. As similar investigations are the subject of ongoing work by other authors, we do not pursue this connection further here.

Inspired by the definitions from \cite{BOORS:2016}, we then introduce an additional stability condition for squares categories (\cref{defn:stable squares cat}). The resulting \textit{stable squares categories} behave similarly to the stable (pointed) double categories of \cite{BOORS:2016}, but with weak equivalences as in a Waldhausen category. Every stable pointed double category in the sense of \cite{BOORS:2016} is a stable squares category whose weak equivalences are equalities (\cref{prop:stable dblcat is stable sqcat}). 

A stable squares category $\dcat C$ is \textit{isostable} when these weak equivalences are invertible and $\dcat C$ satisfies an additional double-categorical condition that ensures $S^{\square}_\bullet \dcat C$ is a simplicial groupoid (see \cref{defn:isostable} and \cref{lem:isostable gives gpd}). Examples include polytopes and isometries and finitely-generated projective $R$-modules and isomorphisms. Mirroring the setting for proto-exact categories \cite[Section 2.4]{dyckerhoff/kapranov:19}, we show that the $S^{\square}_\bullet$-construction of an isostable squares category produces a $2$-Segal space.

\begin{introtheorem}[\cref{thm:2Segal from squares}]
    If $\cat C$ is an isostable squares category, then $[n]\mapsto BS^{\square}_n\cat C$ is a $2$-Segal space.
\end{introtheorem}

In particular, isostable squares categories are proto-Waldhausen, and so the $S^{\square}_\bullet$-construction provides an alternative model for their squares $K$-theory. We can then say that there is a model for the algebraic $K$-theory of a stable squares category with isomorphisms which produces a 2-Segal space. 



\subsection*{Acknowledgements} We are delighted to acknowledge that this work began at the ``Higher Segal Spaces and their Applications to Algebraic $K$-Theory, Hall Algebras, and Combinatorics'' workshop at the Banff International Research Station, where we first started to wonder about the connection between $2$-Segal spaces and squares $K$-theory. We would like to thank Julie Bergner and Cary Malkiewich for helpful conversations and feedback, as well as the anonymous referee for their comments which greatly improved the paper. The first-named author was partially supported by NSF grant DGE-1845298 and the second-named author was partially supported by NSF grant DMS-2506116.


\section{An $S_\bullet$-construction for squares categories}

Waldhausen's $S_\bullet$-construction \cite{waldhausen:1983} takes as input a \textit{Waldhausen category}, which is a pointed category with a notion of cofibration ($\cof$) and weak equivalence ($\xrightarrow{\sim}$) satisfying certain compatibility axioms. In particular, pushouts along cofibrations exist, meaning that certain spans can be completed to squares\[
\begin{tikzcd}
    A \ar[r, >->] \ar[d] & B\\
    C &
\end{tikzcd} \leadsto \begin{tikzcd}
    A \ar[r, >->] \ar[d] & B\ar[d] \\
    C \ar[r, >->] & C\cup_A B
\end{tikzcd}. 
\] Taking $C=*$ to be the zero object, this condition implies that every cofibration $A\cof B$ admits a quotient $B/A = *\cup_A B$. The $S_\bullet$-construction of a Waldhausen category $\cat C$ produces a simplicial category whose objects are sequences of cofibrations $*\cof C_1\cof \dots \cof C_n$ along with compatible choices of quotients $C_{ij}$ for each $C_i\cof C_j$. Morphisms are natural transformations that are pointwise valued in weak equivalences, and $K(\cat C)$ is defined to be the based loop space of the realization of this simplicial category. A characterizing feature of the $S_\bullet$-construction is that it splits the cofiber sequences in $\cat C$; this feature can be seen directly at the level of $K_0(\cat C)= \pi_0(K(\cat C))$ where $[B] = [A] + [B/A]$ whenever $A\cof B$, and more abstractly via Waldhausen's additivity theorem \cite{waldhausen:1983} and the universal property captured in \cite{BGT:13}.

There is another model for the $K$-theory of a Waldhausen category $\cat C$, called the \textit{Thomason construction} (see \cite[Section 1.3]{waldhausen:1983}), which is constructed by associating $\cat C$ to a certain double category whose horizontal morphisms are cofibrations, vertical morphisms are any morphisms in $\cat C$, and squares are pushouts up to weak equivalence, i.e. commutative diagrams \[
\begin{tikzcd}
    A \ar[r, >->] \ar[d] & B\ar[d] \\
    C \ar[r, >->] & D
\end{tikzcd}
\] such that the induced map $C\cup_A B\to D$ is a weak equivalence. In the Thomason construction, the three term relation $[B] = [A] + [B/A]$ is replaced with a four term relation $[D] = [B] + [C] - [A]$. 

The idea of using a $K$-theory construction that encodes four term relations (or an inclusion-exclusion principle more generally) is the impetus for the introduction of \textit{squares categories} \cite{CKMZ:squares} and their $K$-theory. A squares category is a type of double category, meaning that there are two different kinds of morphisms, horizontal ($\cof$) and vertical ($\quot$), as well as squares \[\begin{tikzcd}
    A \ar[r, >->] \ar[d, ->>] \ar[rd, phantom, "\square"] & B\ar[d, ->>] \\
    C \ar[r, >->] & D
\end{tikzcd}
\]which are $2$-cells that encode the interaction between the two different kinds of morphisms. The $K$-theory of a squares category is the loop space of the realization of its double nerve, based at some distinguished object $O$. Under mild assumptions, the connected components of this space naturally encode the four-term relation $[D] = [B] + [C] - [A]$, just as in the Thomason construction. 

Four-term relations arise naturally in many examples of interest in scissors congruence $K$-theory, such as total scissors congruence of polytopes, cut-and-paste groups of manifolds \cite{hoekzema/merling/murray/rovi/semikina:2021}, and certain versions of the Grothendieck spectrum of varieties. Many cases of interest arise from an ambient $1$-category $\cat C$ by selecting certain maps in $\cat C$ to go in the horizontal and vertical directions and specifying the data of certain squares; this notion is called the \textit{squares category generated by $\cat C$} in \cite{CKMZ:squares} and a \textit{category with squares} in \cite{hoekzema/merling/murray/rovi/semikina:2021}. One benefit of the formalism of double categories is to account for other examples (e.g. when the vertical morphisms are a subset of morphisms in $\cat C^{\op}$).

In this section, we define a version of $S_\bullet$-construction for squares categories (\cref{defn:S-dot of cat w squares}), which we denote $S^{\square}_\bullet$. Our definition is a double-categorical version of  Waldhausen's construction, but also encompasses Campbell's variation for subtractive Waldhausen categories \cite{campbell}. In general, the $S^{\square}_\bullet$-construction of a squares category will not model its $K^{\square}$-theory, unless there is an underlying ``three-term''-ness; we identify sufficient conditions for the $K^{\square}$-theory of a squares category to be modeled by the $S^{\square}_\bullet$-construction, following the Waldhausen-Thomason comparison (\cref{thm:square and S comparison}). These kinds of squares categories, which we call \textit{pseudo-Waldhausen categories} (\cref{defn:protoWald cat}), are the basis for the stable squares categories that we introduce in \cref{sec:segal connection}.

\subsection{Squares categories}
We first recall the definition of a squares category, along with the corresponding $K$-theory construction. More details, as well as examples, can be found in \cite{CKMZ:squares}. Squares categories rely on the notion of a double category, so we start with a brief recollection of these 2-dimensional structures. However, we will assume a certain familiarity with double categories, and direct the reader to \cite[Chapter 3]{grandis} for detailed definitions.

\begin{definition}
    A \emph{double category} is an internal category to $\Cat$, the category of small categories and functors.  More explicitly, a double category $\mathbb{A}$ consists of objects $A, B, A', B',\dots$, horizontal morphisms $A\cof B$, vertical morphisms $A\quot A'$, and squares
    \[
    \begin{tikzcd}
        A \ar[r, >->] \ar[d, ->>] \ar[rd, phantom, "\Downarrow"]& B \ar[d, ->>]\\
        A' \ar[r, >->]& B'
    \end{tikzcd}
    \] together with associative and unital compositions for horizontal morphisms, vertical morphisms, and squares. A \textit{double functor} $F \colon\mathbb{A}\to\mathbb{B}$ is an assignment on objects, horizontal morphisms, vertical morphisms, and squares that preserves all compositions and identities strictly.
\end{definition}

Note that each double category $\dcat A$ has an \emph{underlying horizontal category} $\cat H_{\dcat A}$ obtained by considering the objects and horizontal morphisms in $\dcat A$, and an \emph{underlying vertical category} $\cat V_{\dcat A}$ obtained by considering the objects and vertical morphisms in $\dcat A$.

\begin{definition}
    A double category is \emph{flat} if the squares are uniquely determined by their boundary.
\end{definition}

\begin{notation}
    Since all double categories in this paper are flat, we will use the symbol
    \[
    \begin{tikzcd}
        A \ar[r, >->] \ar[d, ->>]\ar[rd, phantom, "\square"] & B \ar[d, ->>]\\
        A' \ar[r, >->]& B'
    \end{tikzcd}
    \] to indicate that the given boundary determines a square in the double category. This is in line with the fact that being a square is property rather than data, and is the notation used in \cite{CKMZ:squares}.
\end{notation}

We now recall the notions of horizontal and vertical natural transformations, specialized to the case of flat double categories. A general definition requires additional coherences which are automatic in the flat case, and can be found in \cite[\S 3.2.7]{grandis}.

\begin{definition}\label{def:naturaltr}
Let $F,G\colon\dcat A\to\dcat B$ be two double functors between flat double categories. A \emph{vertical natural transformation} $\tau\colon F\Rightarrow G$ is the data of a vertical morphism  $\tau_A\colon F(A)\twoheadrightarrow G(A)$ in $\dcat B$ for each object $A\in\dcat A$, such that for each horizontal morphism $f\colon A\cof A'$ in $\dcat A$ the boundary below left determines a square in $\dcat B$,
\[
    \begin{tikzcd}
        F(A) \ar[r, >->,"Ff"] \ar[d, ->>,"\tau_A"']\ar[rd, phantom, "\square"] & F(A') \ar[d, ->>,"\tau_{A'}"]\\
        G(A) \ar[r, >->,"Gf"']& G(A)'
    \end{tikzcd}
    \hspace{3cm}
     \begin{tikzcd}
        F(A) \ar[r, ->>,"Fu"] \ar[d, ->>,"\tau_A"'] & F(A') \ar[d, ->>,"\tau_{A'}"]\\
        G(A) \ar[r, ->>,"Gu"']& G(A)'
    \end{tikzcd}
    \]
and such that for each vertical morphism $u\colon A\twoheadrightarrow A'$ in $\dcat A$ the diagram above right commutes in the underlying vertical category of $\dcat B$.

We denote by $\Fun^v(\dcat A,\dcat B)$ the category of double functors from $\dcat A$ to $\dcat B$ and vertical natural transformations. Dually, one can define a \emph{horizontal natural transformation}, and obtain a category $\Fun^h(\dcat A,\dcat B)$.
\end{definition}

\begin{definition}
    A double category $\dcat C$ is \textit{pointed} if there is a distinguished object $O$ which is initial in $\cat H_{\dcat C}$ and terminal in $\cat V_{\dcat C}$. A morphism of pointed double categories is a double functor that preserves the distinguished object.
\end{definition}

We can now recall the notion of a squares category. 

\begin{definition}\label{defn:squarescat}
    A \emph{squares category} is a flat, pointed double category and a functor of squares categories is a pointed double functor.
\end{definition}

\begin{remark}\label{rmk:O choices}
    The reader familiar with the original definition introduced in \cite{CKMZ:squares} may recall that they require the basepoint $O$ to be initial in both the horizontal and vertical directions. The two definitions of squares categories are not equivalent categorically, but they are in a $K$-theoretical sense. Indeed, as we will see below, the $K$-theory space of a squares category is obtained through a nerve construction which is invariant with respect to reversing the direction of the horizontal or vertical maps. Hence, any squares category in the sense of \cite{CKMZ:squares} can be turned into a squares category in the sense of \cref{defn:squarescat} by taking the opposite vertical category (and vice versa), and both structures will produce the same $K$-theory. 
\end{remark}

\begin{definition}
    Given a squares category $\dcat C$, we denote by $T_n \dcat C$ the category whose objects are sequences of composable horizontal morphisms in $\dcat C$
    \[C_{0}\cof C_{1}\cof C_{2}\cof \dots\cof C_{n}\]
    and whose morphisms are sequences of horizontally composable squares in $\dcat C$ as follows
    \[\begin{tikzcd}
        C_{0}\rar[>->] \dar[->>]\ar[rd, phantom, "\square"]  & C_{1}\rar[>->] \dar[->>]\ar[rd, phantom, "\square"]  & C_{2}\rar[>->] \dar[->>] \ar[rd, phantom, "\square"]& \cdots  \rar[>->]\ar[rd, phantom, "\square"]& C_{n}\dar[->>]\\
        C'_{0}\rar[>->]  & C'_{1}\rar[>->]& C'_{2}\rar[>->]& \cdots  \rar[>->]& C'_{n}
    \end{tikzcd}.\]
    These assemble into a simplicial category $T_\bullet\dcat C$, where all faces and degeneracies behave like the ones in a nerve construction.
\end{definition}

\begin{definition}\label{defn:T-dot}
The \textit{$K$-theory space} of a squares category $\dcat C$ is \[
K^\square(\dcat C) = \Omega_O \abs{T_\bullet(\dcat C)}
,
\] and the \textit{$K$-groups} of $\dcat C$ are the homotopy groups of $K^\square(\dcat C)$ \[K_i^\square (\dcat C) = \pi_i(K^\square(\dcat C)).\]
\end{definition}

\begin{remark}
The construction of $T_\bullet\dcat C$ presented above is precisely the \emph{horizontal nerve} of the underlying double category $\dcat C$. In \cite{CKMZ:squares}, the authors define the $K$-theory space of a squares category using the \emph{double nerve} of $\dcat C$ instead: the bisimplicial set whose $(m,n)$-simplices are grids of squares in $\dcat C$ of the form \[
 \begin{tikzcd}
 C_{00} \ar[r, >->] \ar[d, ->>] \ar[rd, phantom, "\square"] & \dots \ar[r, >->] & C_{0n} \ar[d, ->>]\\
 \vdots \ar[d, ->>] & \ddots \ar[rd, phantom, "\square"] & \vdots\ar[d, ->>]\\
 C_{m0} \ar[r, >->]& \dots\ar[r, >->] & C_{mn}
 \end{tikzcd}
 \] All nerves of double categories produce the same space up to homeomorphism, so these constructions are equivalent at the $K$-theory level; our choice to present the space via a simplicial category instead of a bisimplicial set is motivated by our ultimate goal of comparing it to an $S_\bullet$-construction.
\end{remark}

\begin{remark}\label{rmk:Tdot1segalincat}
Note that the simplicial category $[n]\mapsto T_n\dcat C$ is a $1$-Segal object in $\Cat$ with the canonical model structure. However, the simplicial space $[n]\mapsto BT_n\dcat C$ may not be $1$-Segal as taking classifying spaces will generally not preserve homotopy pullbacks. Additionally, we note that the double nerve is neither a $1$-Segal set nor a $2$-Segal set in general.
\end{remark}

Under an additional condition we recover the expected description of $K_0^\square(\dcat C)$ as a Grothendieck group.

\begin{theorem}[{\cite[Theorem 3.1]{CKMZ:squares}}]
    Suppose a squares category $\dcat C$ is such that for any two objects $A,B$ there is an object $X$ and squares \[
    \begin{tikzcd}
        A\ar[r, >->]\ar[d, ->>] \ar[rd, phantom, "\square"] & X \ar[d, ->>] \\
        O \ar[r, >->] & B
    \end{tikzcd} ~\text{ and }~ \begin{tikzcd}
        B\ar[r, >->]\ar[d, ->>] \ar[rd, phantom, "\square"] & X \ar[d, ->>] \\
        O \ar[r, >->] & A
    \end{tikzcd}.
    \] Then $K_0^{\square}(\dcat C)\cong \ZZ[\ob \dcat C]/\sim$ where $[O]\sim 0$ and $[A] + [D] = [B] + [C]$ for every square\[
    \begin{tikzcd}
        A\ar[r, >->]\ar[d, ->>] \ar[rd, phantom, "\square"] & B \ar[d, ->>] \\
        C \ar[r, >->] & D
    \end{tikzcd}.
    \]
\end{theorem}

With an eye towards developing a $S_\bullet$-construction for squares categories, we emphasize that every Waldhausen category naturally admits the structure of a squares category.
Other examples of squares categories are given in \cref{sec:examples}.

\begin{example}[{\cite[Example 1.8]{CKMZ:squares}}]\label{ex:waldhausen}
    Any Waldhausen category $\cat C$ gives rise to a squares category, in which the objects are the objects of $\cat C$, the horizontal morphisms are the cofibrations in $\cat C$, the vertical morphisms are all the morphisms in $\cat C$, and where a commutative diagram
    \[
    \begin{tikzcd}
        A\ar[r, >->]\ar[d, ->>]  & B \ar[d, ->>] \\
        C \ar[r, >->] & D
    \end{tikzcd}
    \]
is a square in the double category if the induced map from the pushout $B\sqcup_A C\to D$ is a weak equivalence. The basepoint $O$ is the zero object of $\cat C$.

The $K$-theory of this squares category is precisely the Thomason model for the $K$-theory of the Waldhausen category $\cat C$, which agrees with the space obtained through the classical $S_\bullet$-construction as explained in \cite[Section 1.3]{waldhausen:1983}.
\end{example}

\subsection{The $S^{\square}_\bullet$-construction}\label{section:Sdot}

Every squares category admits a version of the $S_\bullet$-construction, obtained by simply considering diagrams of the appropriate shape. However, in general this construction may not model their $K$-theory space as given in \cref{defn:T-dot}. In this section we explain how to extract a class of weak equivalences from the data present in a squares category, and we define an $S^{\square}_\bullet$-construction analogous to the one for Waldhausen categories. Similar constructions have been defined in other, related settings \cite{BOORS:2016, camzak:22, dyckerhoff/kapranov:19, SS-CGW}. In later subsections we will discuss the requirements needed on a squares category to ensure the two simplicial constructions agree after realization.

\begin{notation}
For any $n\geq 0$, let $S^{\square}_n$ denote the (flat) double category free on the data
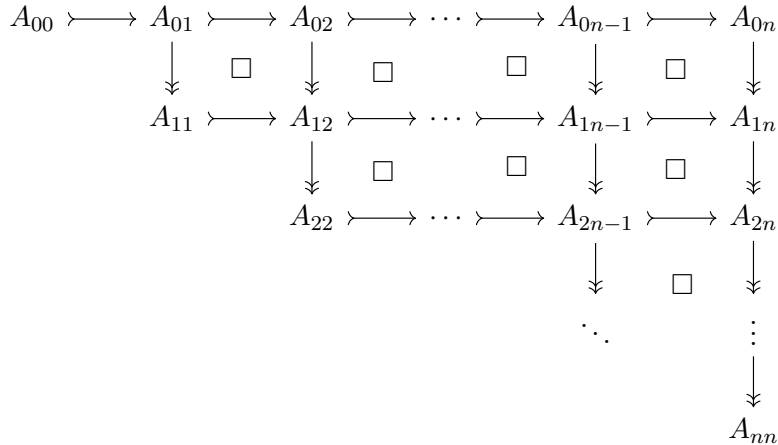
\[
\begin{tikzcd}
        A_{00} \ar[r, >->] & A_{01} \ar[r, >->]\ar[d, ->>] \ar[rd, phantom, "\square"] & A_{02} \ar[r, >->] \ar[d, ->>] \ar[rd, phantom, "\square"] & \cdots \ar[r, >->]  \ar[rd, phantom, "\square"] & A_{0n-1} \ar[r, >->] \ar[d, ->>] \ar[rd, phantom, "\square"] & A_{0n} \ar[d, ->>] \\
    & A_{11} \ar[r, >->] & A_{12} \ar[r, >->] \ar[d, ->>] \ar[rd, phantom, "\square"] & \cdots \ar[r, >->] \ar[rd, phantom, "\square"] & A_{1n-1} \ar[r, >->] \ar[d, ->>] \ar[rd, phantom, "\square"] & A_{1n} \ar[d, ->>]\\
    && A_{22} \ar[r, >->] & \cdots \ar[r, >->]  & A_{2n-1} \ar[r, >->] \ar[d, ->>] \ar[rd, phantom, "\square"] & A_{2n} \ar[d, ->>]\\
    &&& & \ddots & \vdots \ar[d, ->>]\\
    &&&& & A_{nn}
    \end{tikzcd}
    \] A free double category is analogous to the usual categorical notion, and we refer the reader to \cite[Section 3]{fiore/paoli/pronk} for a formal discussion.
\end{notation}

\begin{definition}\label{defn:S-dot of cat w squares}
    Given a squares category $\dcat C$ and $n\geq 0$, we define $S^{\square}_n\dcat C$ as the subcategory of $\Fun^v(S^{\square}_n, \dcat C)$ whose objects are the functors $F\colon S^{\square}_n\to\dcat C$ such that $F(A_{ii})=O$ for all $i$, and morphisms are vertical natural transformations. Then $S^{\square}_\bullet \dcat C$ forms a simplicial category, where the face map $d_i$ removes the $i$th row and column (and composes when appropriate) and degeneracy maps simply insert identity maps and identity squares.
\end{definition}

It is straightforward to check that this definition extends to a functor between squares categories and simplicial categories.

\begin{proposition}
    The $S^{\square}_\bullet$-construction defines a functor $S^{\square}_\bullet\colon {\rm SqCat}\to s\Cat$.
\end{proposition}

\begin{remark}\label{rmk:augmented Sdot}
In some examples, it seems natural to relax the condition in \cref{defn:S-dot of cat w squares} that $F(A_{ii}) = O$ and instead ask that $F(A_{ii})\in \cat A$ for all $i$, where $\cat A$ is some distinguished collection of objects. One could define a notion of \textit{augmented squares category}, similar to an \textit{augmented double category} described in \cite[Definition 3.6 and Proposition 3.8]{BOORS:2016}, and define an $S_\bullet$-construction for such squares categories. 

Although we do not pursue this idea in this paper, one can see how it is the natural framework to capture certain examples. For instance, one can encode the $K$-theory of graphs considered in \cite{BCGP:24} using an augmented squares category, by letting horizontal morphisms be subgraph inclusions which are bijections on vertices, vertical morphisms be maps which contract some collection of edges and are bijections on connected components, and squares be pushouts. The different notions of weak equivalence outlined in \cite[Definition 5.4]{BCGP:24} can be incorporated into a squares category by asking for squares to be pushouts up to the corresponding notion of weak equivalence. This squares category does not have a distinguished initial/terminal object, but instead, a collection of such objects (the discrete sets). 
\end{remark}

\begin{example}\label{ex:S2}
    Let us depict explicitly what the objects and morphisms in the category $S^{\square}_2\dcat C$ look like. An object consists of a diagram in $\CC$ of the form
    
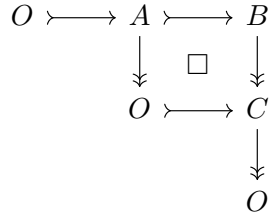
\[
\begin{tikzcd}
        O \ar[r, >->] & A \ar[r, >->]\ar[d, ->>] \ar[rd, phantom, "\square"] & B  \ar[d, ->>] \\
    & O \ar[r, >->] & C  \ar[d, ->>] \\
    && O
    \end{tikzcd}
    \]
    A morphism consists of a diagram
       
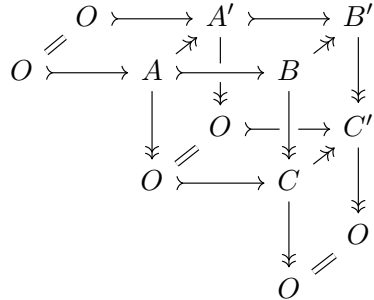
\[
\begin{tikzcd}[row sep=tiny, column sep=tiny]
& O \ar[rr,>->] & & A'\ar[rr, >->]\ar[dd, ->>]  & & B'\ar[dd, ->>]\\
        O\ar[ur,equal] \ar[rr, >->] & & A\ar[ur,->>] \ar[rr, >->, crossing over]\ar[dd, ->>]  & & B \ar[ur,->>] &  \\
       & & & O \ar[rr, >->] & & C'  \ar[dd, ->>]  \\
    & & O\ar[ur,equal] \ar[rr, >->] & & C \ar[ur,->>] \ar[dd, ->>]\ar[from=uu,->>, crossing over] & \\
       & & & & & O \\
    & & & & O\ar[ur,equal] &
    \end{tikzcd}
    \]  where any square whose boundary has two horizontal and two vertical morphisms is a square in $\dcat C$, and any square whose boundary consists of four vertical morphisms is a commutative diagram in the underlying vertical category of $\dcat C$.
\end{example}

\begin{remark}
    Unlike the classical $S_\bullet$-construction for Waldhausen categories, there is a choice required in our definition of $S^{\square}_\bullet\dcat C$: since we do not assume that there is an ambient (1-)category from which $\dcat C$ arises, we can no longer speak of ``natural transformations'', and instead must decide to use either horizontal or vertical natural transformations. The theory could be analogously developed for horizontal ones; our vertical bias is due to the fact that this is what fits many motivating examples. We will see a similar bias again in \cref{defn:protoWald cat}.
    Alternatively, one could construct a simplicial double category whose horizontal (resp.\ vertical) morphisms are the horizontal (resp.\ vertical) natural transformations; this would allow the $S_\bullet^{\square}$-construction to be iterable, and indeed is the approach taken in \cite{SS-CGW}, but is more technically involved than what we require here.
\end{remark}

\begin{remark}
    The definition above is what motivates our requirement in \cref{defn:squarescat} that $O$ be terminal in vertical morphisms instead of initial. In many examples of interest the distinguished object is initial in vertical morphisms, but as previously mentioned in \cref{rmk:O choices}, this problem can be circumvented by taking the opposites of vertical arrows.
\end{remark}

In the classical $S_\bullet$-construction for Waldhausen categories, one must restrict the natural transformations to those that are pointwise weak equivalences in order to construct the $K$-theory space. To make an analogous restriction possible in our setting, we now identify a notion of horizontal and vertical weak equivalence, encoded by the squares.

\begin{definition}\label{defn:weakequiv}
Let $\dcat C$ be a squares category.    A horizontal morphism $A\cof B$ is a \textit{horizontal weak equivalence} if it participates in a square\[
    \begin{tikzcd}
        A \ar[r, >->] \ar[d, ->>]\ar[rd, phantom, "\square"] & B \ar[d, ->>] \\
        O \ar[r, equal] & O
    \end{tikzcd}
    \]  Similarly, define \textit{vertical weak equivalences} to be the vertical morphisms $A\quot B$ that participate in a square\[
    \begin{tikzcd}
        O \ar[r, >->] \ar[d, equal]\ar[rd, phantom, "\square"] & A \ar[d, ->>] \\
        O \ar[r, >->] & B
    \end{tikzcd}
    \]
    Denote these collections of morphisms by $w^h\dcat C\subseteq \cat H_{\dcat C}$ and $w^v\dcat C\subseteq \cat V_{\dcat C}$, respectively.
\end{definition}

\begin{remark}\label{rmk:weakequivsforfree}
By definition, the pointwise vertical maps appearing in a morphism in the category $S^{\square}_\bullet\dcat C$ are all in $w^v\dcat C$. To illustrate this in the case of $n=2$ of \cref{ex:S2}, we see that $A\twoheadrightarrow A'$ and $C\twoheadrightarrow C'$ are weak equivalences directly by definition, and the composite \[\begin{tikzcd}
              O \ar[r, >->] \ar[d, equal] \ar[rd, phantom, "\square"] &  A \ar[r, >->]\ar[d, ->>] \ar[rd, phantom, "\square"] & B \ar[d, ->>] \\
           O \ar[r, >->] & A' \ar[r, >->] & B'
\end{tikzcd}
    \] gives the square required for $B\twoheadrightarrow B'$ to be a weak equivalence.
\end{remark}

We now detail how this notion of weak equivalence compares to the classical one for the case of Waldhausen categories, whose structure as a squares category was detailed in \cref{ex:waldhausen}. Other examples are discussed in \cref{sec:examples}.

\begin{example}\label{ex:waldhausenweakequivs}
If $\cat C$ is a Waldhausen category viewed as a squares category as in \cref{ex:waldhausen}, then $w^v\dcat C$ are the weak equivalences in $\cat C$. To see this, note that a vertical morphism (i.e. any morphism $f\colon A\to B$ in $\cat C$) is a vertical weak equivalence precisely if the induced map out of the pushout $A\cong A\sqcup_O O\to B$ is a weak equivalence in $\cat C$, but this map is $f$ itself.

Similarly, one can check that the maps in $w^h\dcat C$ are the cofibrations $f\colon A\cof B$ such that the unique map from the cofiber $B/A\to O$ is a weak equivalence. The Gluing lemma ensures that any trivial cofibration is in $w^h\dcat C$. Conversely, if $\cat C$ satisfies the extension axiom and is such that whenever $X\to 0$ is a weak equivalence, then $0\to X$ is also a weak equivalence (which holds, for instance, if $\cat C$ satisfies the saturation axiom), then any map in $w^h\dcat C$ is a trivial cofibration.
\end{example}

\begin{remark}\label{rmk:weaklyexact}
Note that if $\cat C$ is a Waldhausen category viewed as a squares category, then our definition of $S^{\square}_\bullet\cat C$ does not quite match the simplicial category $wS_\bullet\cat C$ introduced by Waldhausen. The issue is with the objects, which in the classical setting consist of diagrams with pushout squares, but in the double categorical setting have pushouts \emph{up to weak equivalence}. Interestingly, the morphisms are defined in the same manner for both constructions: that is, a natural transformation that is a pointwise weak equivalence encodes the same data as a vertical natural transformation valued in $w^v\dcat C$. Indeed, by the gluing lemma, any commutative square of the form \[
    \begin{tikzcd}
        A \ar[r, >->] \ar[d, "\sim"'] & B \ar[d, "\sim"] \\
        A' \ar[r, >->] & B'
    \end{tikzcd}
    \] will be such that the induced map $A'\sqcup_A B\to B'$ is a weak equivalence and hence it is a square in the double category.
Despite this discrepancy, in \cref{cor:WaldhausenisProtoWaldhausen} we show that $S^{\square}_\bullet\cat C$ does give an alternate model for the $K$-theory of any Waldhausen category; hence, it agrees with Waldhausen's construction after realization.

A feature of the $S^{\square}_\bullet$-construction is that a squares functor $\cat C\to \cat D$ between Waldhausen categories will induce a morphism of simplicial categories $S^{\square}_\bullet \cat C\to S^{\square}_\bullet \cat D$. Notably, the same statement does not hold for Waldhausen's definition of $S_\bullet$. Indeed, although every exact functor of Waldhausen categories induces a squares functor on their associated squares categories (\cref{ex:waldhausen}), it is not the case that every squares functor induces an exact functor, as it may not preserve pushouts along cofibrations in general. Instead, a squares functor $F\colon \cat C\to \cat D$ will only be ``weakly exact'' in the sense that there is a (canonical) weak equivalence $FB\cup_{FA} FC\xrightarrow{\sim} F(B\cup_A C)$ for any span $C\leftarrow A\cof B$. Our definition of $S^{\square}_\bullet$ is functorial in weakly exact functors.
    \end{remark}

\begin{remark}
    Our definition of the $S^{\square}_\bullet$-construction is very similar to the $S'_\bullet$-construction of Blumberg--Mandell. In \cite[Theorem 2.9]{blumberg/mandell:08}, they show that the $S'_\bullet$-construction models the $K$-theory of a large class of Waldhausen categories. Although their construction is morally similar to ours (using homotopy pushouts rather than actual pushouts), we do not in general expect the two to be the same since the $S'_\bullet$-construction uses a more general notion of cofibration.
\end{remark}

We conclude this subsection by showing that, just like in the classical setting, when the weak equivalences are isomorphisms our construction above does not introduce any additional data.

\begin{proposition}\label{prop:S-bullet with isos}
  Let $\dcat C$ be a squares category such that $w^v\dcat C=iso(\cat V_{\dcat C})$. Then $\abs{S^{\square}_\bullet \dcat C}\simeq \abs{\ob S^{\square}_\bullet \dcat C}$.
\end{proposition}\begin{proof}
This is the analogue of \cite[Corollary of Lemma 1.4.1]{waldhausen:1983}. Note that $N_0^wS^{\square}_n\dcat C = \ob S^{\square}_n\dcat C$ and all of the face and degeneracy maps in the nerve $N^w_*S^{\square}_n\dcat C$ are homotopy equivalences (since all the morphisms in the category $S^{\square}_n\dcat C$ are isomorphisms). Hence $\abs{S^{\square}_\bullet \dcat C}\simeq \abs{N_0^wS^{\square}_\bullet \dcat C} \simeq \abs{\ob S^{\square}_\bullet \dcat C}$.
\end{proof}

\subsection{Proto-Waldhausen categories}
In this subsection, we describe sufficient conditions for the $T_\bullet$-construction and $S^{\square}_\bullet$-construction to coincide for a given squares category. Inspired by the proto-exact categories of \cite{dyckerhoff/kapranov:19}, we introduce a notion of \textit{proto-Waldhausen category}. Essentially, this will be a double categorical version of a Waldhausen category, where we enforce conditions on the squares so they behave like pushout squares.

Given a squares category $\dcat C$ and any (flat) double category $\dcat I$, we can define a double category $\Fun(\dcat I,\dcat C)$ whose objects are the double functors, horizontal (resp.\ vertical) morphisms are the horizontal (resp.\ vertical) natural transformations of \cref{def:naturaltr}, and where squares are defined pointwise. Moreover,  $\Fun(\dcat I, \dcat C)$ inherits a squares category structure with pointwise basepoint. We will particularly care about the cases where $\dcat I$ is the flat double category generated by a single square, or a single ``span'', or a single ``cospan'' as illustrated below
\[
    \begin{tikzcd}
\scriptstyle\bullet \ar[r, >->] \ar[d, ->>] \ar[rd, phantom, "\square"] & \scriptstyle\bullet \ar[d, ->>]\\
\scriptstyle\bullet \ar[r, >->, swap]& \scriptstyle\bullet
\end{tikzcd}
\hspace{1cm} ; \hspace{1cm}
    \begin{tikzcd}
\scriptstyle\bullet \ar[r, >->] \ar[d, ->>]  & \scriptstyle\bullet \\
\scriptstyle\bullet &
\end{tikzcd}
\hspace{1cm} ; \hspace{1cm}
    \begin{tikzcd}
 & \scriptstyle\bullet \ar[d, ->>]\\
\scriptstyle\bullet \ar[r, >->, swap]& \scriptstyle\bullet
\end{tikzcd};
    \] we denote these diagrams by $\tinysquare$, $\tinyspan$ and $\tinycospan$, respectively.

\begin{definition}\label{defn:protoWald cat}
A \textit{proto-Waldhausen category} is a squares category $\dcat C$ satisfying the following conditions:
\begin{itemize}
        \item[(i)]\label{protoi} 
        The functor $i^*\colon \Fun^v\left(\tinysquare, \dcat C\right)\to \Fun^v\left(\tinyspan, \dcat C\right)$ induced by the inclusion $i\colon\tinyspan\hookrightarrow\tinysquare$ admits a section functor $s$.
        \item[(ii)]\label{protoii} There is a natural transformation $w\colon si^*\Rightarrow \id$ whose component on a square is
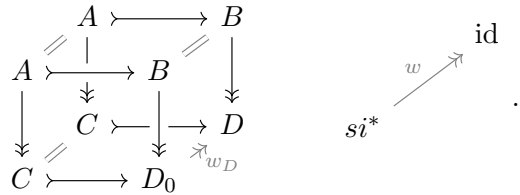
\[
\begin{tikzcd}[row sep=tiny, column sep=tiny]
  & A\ar[rr,>->]\ar[dd,->>] && B\ar[dd,->>]\\
  A\ar[ur,equal, gray]\ar[rr,>->, crossing over]\ar[dd,->>] && B\ar[ur,equal, gray] &\\
  & C\ar[rr,>->] && D\\
   C\ar[ur,equal, gray]\ar[rr,>->] && D_0\ar[ur,->>,gray,"w_D"']\ar[from=uu,->>, crossing over] &\\
\end{tikzcd} \hspace{1cm}
\begin{tikzcd}
  & \id\\
  si^*\ar[ur,->>,gray,"w"]\\
\end{tikzcd}.
\] Recall that, by definition of morphisms in $\Fun^v\left(\tinysquare,\dcat C\right)$, all squares in the above diagram are either commutative squares in the underlying vertical category or distinguished squares, as appropriate.
        \end{itemize}
\end{definition}

The definition above is meant to invoke the idea of squares being pushouts up to weak equivalence, as in the Thomason construction of a Waldhausen category. To make this more clear, and to elucidate the definitions, let us unpack explicitly the conditions in \cref{defn:protoWald cat}.

\begin{remark}  Condition (i) can be interpreted as follows. First, suppose we have an object in $ \Fun^v\left(\tinyspan, \dcat C\right)$, i.e. a span as depicted below left. Then, the action of the section $s$ on objects allows us to complete this span to a distinguished square as below right\[
\begin{tikzcd}
    A \ar[r, >->] \ar[d, ->>] & B\\
    C & 
\end{tikzcd} ~\mapsto~ \begin{tikzcd}
    A \ar[r, >->] \ar[d, ->>] \ar[rd, phantom, "\square"] & B \ar[d, ->>] \\
    C \ar[r, >->] & D_0
\end{tikzcd}.
\] Given a vertical natural transformation between such spans,\[
\begin{tikzcd}[row sep=tiny, column sep=tiny]
  & A'\ar[rr,>->]\ar[dd,->>] && B'\\
  A\ar[dd,->>]\ar[ur,->>]\ar[rr,>->, crossing over] && B\ar[ur,->>] &\\
  & C' && \\
   C\ar[ur,->>] && &\\
\end{tikzcd},
\] the section $s$ additionally supplies a vertical map $D_0\quot D_0'$ to form a diagram as below\[
\begin{tikzcd}[row sep=tiny, column sep=tiny]
  & A'\ar[rr,>->]\ar[dd,->>] && B'\ar[dd,->>]\\
  A\ar[ur,->>]\ar[rr,>->, crossing over]\ar[dd,->>] && B\ar[ur,->>] &\\
  & C'\ar[rr,>->] && D'_0\\
   C\ar[ur,->>]\ar[rr,>->] && D_0\ar[ur,->>,dashed]\ar[from=uu,->>, crossing over] &\\
\end{tikzcd}.
\] This diagram is such that all of the squares are either commutative or distinguished as appropriate. This assignment is functorial in the direction of the vertical natural transformation.

 Given a distinguished square, we can apply $i^*$ to obtain a span and then $s$ to obtain a new, possibly different distinguished square with the same span as the original one, as depicted below\[
\begin{tikzcd}
    A \ar[r, >->] \ar[d, ->>] \ar[rd, phantom, "\square"] & B \ar[d, ->>] \\
    C \ar[r, >->] & D
\end{tikzcd} ~\mapsto~ \begin{tikzcd}
    A \ar[r, >->] \ar[d, ->>] & B\\
    C & 
\end{tikzcd} ~\mapsto~ \begin{tikzcd}
    A \ar[r, >->] \ar[d, ->>] \ar[rd, phantom, "\square"] & B \ar[d, ->>] \\
    C \ar[r, >->] & D_0
\end{tikzcd}.
\] Condition (ii) says that there is a vertical natural transformation between these two distinguished squares, as depicted in \cref{defn:protoWald cat}(ii). In particular, note that the induced vertical morphism $w_D\colon D_0\twoheadrightarrow D$ is always a vertical weak equivalence, since we have a composite of the distinguished squares
\[\begin{tikzcd}
              O \ar[r, >->] \ar[d, equal] \ar[rd, phantom, "\square"] &  C \ar[r, >->]\ar[d, equal] \ar[rd, phantom, "\square"] & D_0 \ar[d, ->>, "w_D"] \\
           O \ar[r, >->] & C \ar[r, >->] & D
\end{tikzcd}.
    \]
The naturality requirement in condition (ii) states that, for every given diagram as below left whose squares are either commutative or distinguished as appropriate \[
\begin{tikzcd}[row sep=tiny, column sep=tiny]
  & A'\ar[rr,>->]\ar[dd,->>] && B'\ar[dd,->>]\\
  A\ar[ur,->>]\ar[rr,>->, crossing over]\ar[dd,->>] && B\ar[ur,->>] &\\
  & C'\ar[rr,>->] && D'\\
   C\ar[ur,->>]\ar[rr,>->] && D\ar[ur,->>,dashed]\ar[from=uu,->>, crossing over] &\\
\end{tikzcd} \hspace{2cm}
\begin{tikzcd}
  D_0\rar[->>]\dar[->>,"w_D"'] & D'_0\dar[->>,"w_{D'}"]\\
  D\rar[->>] & D'
\end{tikzcd}
\] the resulting diagram depicted above right commutes in $\cat V_{\dcat C}$.
        \end{remark}

As expected, every Waldhausen category is proto-Waldhausen, using the squares category structure from \cref{ex:waldhausen}. 
Further examples (and non-examples) of proto-Waldhausen categories are given in \cref{sec:examples}.

\begin{example}\label{ex:waldisprotowald}
If $\cat C$ is a Waldhausen category, then the squares category defined from $\cat C$ is proto-Waldhausen. To check condition (i), note that we can complete any span as below left 
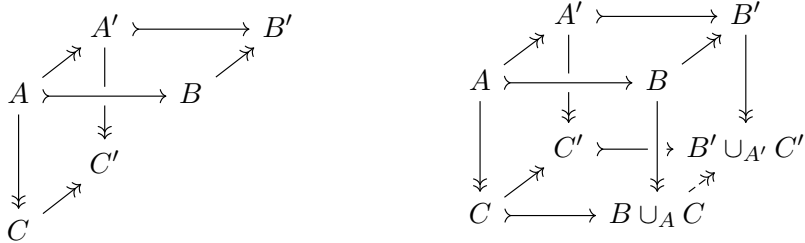
\[
\begin{tikzcd}[row sep=small, column sep=small]
  & A'\ar[rr,>->]\ar[dd,->>] && B'\\
  A\ar[dd,->>]\ar[ur,->>]\ar[rr,>->, crossing over] && B\ar[ur,->>] &\\
  & C' && \\
   C\ar[ur,->>] && &\\
\end{tikzcd} \hspace{2cm}
\begin{tikzcd}[bo column sep, bo row sep]
  & A'\ar[rr,>->]\ar[dd,->>] && B'\ar[dd,->>]\\
  A\ar[ur,->>]\ar[rr,>->, crossing over]\ar[dd,->>] && B\ar[ur,->>] &\\
  & C'\ar[rr,>->] && B'\cup_{A'} C'\\
   C\ar[ur,->>]\ar[rr,>->] && B\cup_A C\ar[ur,->>,dashed] \ar[from=uu,->>, crossing over]&\\
\end{tikzcd}
\] to a diagram as above right by taking pushouts. The required vertical map is given by the universal property of the pushout for $B\sqcup_A C$. By construction, the front and back faces of the cube are squares in the double category, and the right face in the cube is a commutative diagram. To see that the bottom face we obtain is a square in the double category, note that there is a weak equivalence \[(B\cup_A C)\cup_C C'\cong B\cup_A C'\cong (B\cup_A A')\cup_{A'} C'\xrightarrow{\sim} B'\cup_{A'} C'\] where the last map uses the gluing axiom and the weak equivalence $B\cup_A A'\xrightarrow{\sim} B'$ as the top face of the cube is a square in the double category. Functoriality of this section is guaranteed by the universal property of the pushout.

The natural transformation $w$ of condition (ii) is constructed using the definition of the squares in the squares category obtained from $\cat C$. Finally, the naturality of $w$ is a direct consequence of the universal property of the pushout.
\end{example}

\subsection{A Waldhausen--Thomason comparison}

In this section, we show that the constructions $T_\bullet \dcat C$ and $S^{\square}_\bullet \dcat C$ agree after realization whenever $\dcat C$ is proto-Waldhausen, following Waldhausen's strategy (see \cite[\S 1.3]{waldhausen:1983}). These simplicial categories are not directly connected through a simplicial map. Instead, the key is to construct a third simplicial category $T^+_\bullet\dcat C$ together with simplicial maps $$T_\bullet\dcat C \leftarrow T^+_\bullet \dcat C\to S^{\square}_\bullet\dcat C$$ which induce homotopy equivalences after realization.

We start by introducing the auxiliary simplicial category $T^+_\bullet\dcat C$. Intuitively, its role is to extend the objects of $T_\bullet\dcat C$ to include choices of ``cofibers''. In the classical setting of Waldhausen categories, these are constructed as actual cofibers by taking sequential pushouts. In our setting, the role played by pushouts squares is replaced by the squares in the squares category.

\begin{definition}\label{defTplus}
    Given a squares category $\dcat C$, let $T^+_n\dcat C$ denote the category whose objects are diagrams of the form
    \begin{equation}\label{eqn1}\begin{tikzcd}
        C_{0}\ar[rd, phantom, "\square"] \ar[d, ->>]\ar[r, >->] & C_{1} \ar[r, >->]\ar[d, ->>] \ar[rd, phantom, "\square"] & C_{2} \ar[r, >->] \ar[d, ->>] \ar[rd, phantom, "\square"] & \cdots \ar[r, >->]  \ar[rd, phantom, "\square"]  & C_{n} \ar[d, ->>] \\
    O \ar[r, >->]& C_{01}\ar[rd, phantom, "\square"]\ar[d, ->>] \ar[r, >->] & C_{02} \ar[r, >->] \ar[d, ->>] \ar[rd, phantom, "\square"] & \cdots \ar[r, >->] \ar[rd, phantom, "\square"]  & C_{0n} \ar[d, ->>]\\
    &O\ar[r, >->] & C_{12} \ar[r, >->] & \cdots \ar[r, >->]  & C_{1n} \ar[d, ->>]\\
    &&&  \ddots & \vdots \ar[d, ->>]\\
    &&&&  O
    \end{tikzcd}
    \end{equation} and whose morphisms are vertical natural transformations between these diagrams; that is, pointwise vertical maps such that any square that is formed with two horizontal and two vertical boundaries is a square in $\dcat C$, and any square with four vertical boundaries is a commutative square in the underlying vertical category $\cat V_{\dcat C}$.

    One can check that these assemble into a simplicial category $T^+_\bullet\dcat C$, where the face map $d_i$ removes the $i$th row and column (and composes when appropriate) and degeneracy maps simply insert identity maps and identity squares.
\end{definition}

\begin{proposition}\label{prop:TplustoT}
    If $\dcat C$ is proto-Waldhausen, then the forgetful map $U\colon T^+_\bullet\dcat C\to T_\bullet\dcat C$ is a map of simplicial categories which is a homotopy equivalence after realization.
\end{proposition}
\begin{proof}
    For each $n$, the map $U_n$ takes an object in $T^+_n\dcat C$ (that is, a diagram as in \ref{eqn1}) to its top row. This is clearly functorial, and it is straightforward to verify that it assembles into a map of simplicial categories.

    We now define a section functor $F_n\colon T_n\dcat C\to T^+_n\dcat C$ for each $n$. This assignment takes an object \[C_0\cof C_1\cof \cdots\cof C_n\] in $T_n\dcat C$ to the diagram
    \[\begin{tikzcd}
        C_{0}\ar[rd, phantom, "\square"] \ar[d, ->>]\ar[r, >->] & C_{1} \ar[r, >->]\ar[d, ->>] \ar[rd, phantom, "\square"] & \cdots \ar[r, >->]  \ar[rd, phantom, "\square"]  & C_{n} \ar[d, ->>] \\
    O \ar[r, >->]& D_{01}\ar[rd, phantom, "\square"]\ar[d, ->>] \ar[r, >->]  & \cdots \ar[r, >->] \ar[rd, phantom, "\square"]  & D_{0n} \ar[d, ->>]\\
    &O\ar[r, >->] & \cdots \ar[r, >->]  & D_{1n} \ar[d, ->>]\\
    &&  \ddots & \vdots \ar[d, ->>]\\
    &&&  O
    \end{tikzcd}
    \] The rest of the data here is constructed using condition (i) of \cref{defn:protoWald cat}\footnote{In the definition, the existence of a section $s$ is a property rather than additional data, so a priori there could be several choices of $s$ to complete the spans; each such choice gives rise to a different section $F_n$.} to complete the spans sequentially as follows:
    \[\begin{tikzcd}
        C_0 \ar[rd, phantom, "\square"]\ar[d, ->>]\ar[r, >->] & C_1\ar[d, ->>, dashed]\\
        O \ar[r, >->, dashed] & D_{01}\\
    \end{tikzcd} \ ; \
    \begin{tikzcd}
        C_i \ar[rd, phantom, "\square"]\ar[d, ->>]\ar[r, >->] &  C_{i+1}\ar[d, ->>, dashed]\\
        D_{0 i} \ar[r, >->, dashed] & D_{0i+1}\\
    \end{tikzcd} \ ; \
    \begin{tikzcd}
        D_{i, i+1}\ar[rd, phantom, "\square"] \ar[d, ->>]\ar[r, >->] &  D_{i, i+2}\ar[d, ->>, dashed]\\
        O \ar[r, >->, dashed] & D_{i+1,i+2}\\
    \end{tikzcd} \ ; \
    \begin{tikzcd}
        D_{i,j} \ar[rd, phantom, "\square"]\ar[d, ->>]\ar[r, >->] & D_{i,j+1}\ar[d, ->>, dashed]\\
        D_{i+1,j} \ar[r, >->, dashed] & D_{i+1,j+1}\\
    \end{tikzcd}.\] A morphism of sequences in $T_n\dcat C$ induces a morphism in $T^+_n\dcat C$ between these span completions, and this is functorial as a consequence of the functoriality of the span completions from condition (i).

    Clearly $U_n F_n=\id$. To conclude our result, it suffices to construct a natural transformation $\tau\colon F_n U_n\Rightarrow\id$ for each $n$, as this natural transformation will realize to a homotopy. Given an object $C\in T^+_n\dcat C$ as in diagram (\ref{eqn1}), we need to construct the data of a vertical natural transformation
    \[
\begin{tikzcd}[row sep=tiny, column sep=tiny]
                              & C_0 \ar[rr,>->] \ar[dd, ->>]&                                            & C_1\ar[rr, >->]\ar[dd, ->>]   &                   & \cdots\ar[rr, >->] &                               & C_n\ar[dd, ->>]\\
C_0\ar[ur,equal,gray] \ar[rr, >->, crossing over]\ar[dd, ->>] &                 & C_1\ar[ur,equal,gray] \ar[rr, >->, crossing over]  &                               & \cdots\ar[rr, >->] &                   & C_n \ar[ur,equal,gray] &  \\
                              & O   \ar[rr,>->] &                                            & C_{01}\ar[rr, >->]\ar[dd, ->>]&                   & \cdots\ar[rr, >->] &                               & C_{0n}  \ar[dd, ->>]  \\
O\ar[ur,equal,gray]  \ar[rr, >->]              &                 & D_{01}\ar[rr, >->, crossing over]\ar[dd, ->>] \ar[from=uu,->>, crossing over]            &                               & \cdots\ar[rr, >->] &                   & D_{0n} \ar[from=uu,->>, crossing over] & \\
                              &                 &                                            & O \ar[rr, >->]                &                   & \cdots\ar[rr, >->] &                               & C_{1n} \ar[dd, ->>] & & C\\
                              &                 & O\ar[rr, >->] \ar[ur,equal]                &                               & \cdots\ar[rr, >->] &                   & D_{1n}\ar[dd, ->>]\ar[from=uu,->>, crossing over]&  & F_n U_n (C)\ar[ur,->>,"\tau_C",gray] &\\
                              &                 &                                            &                               &                   &    \ddots         &                               & \vdots\ar[dd, ->>]\\
                              &                 &                                            &                               &    \ddots         &                   &  \vdots\ar[dd, ->>]           & \\
                              &                 &                                            &                               &                   &                   &                               & O\\
                              &                 &                                            &                               &                   &                   &      O  \ar[ur,equal,gray]         & \\
    \end{tikzcd}
    \] We define the required vertical maps inductively, starting left to right in the top row and then moving on to the next. To illustrate an arbitrary step, suppose that we have constructed the data
    \begin{equation}\label{eqn2}\begin{tikzcd}[row sep=tiny,column sep=tiny]
        & C_{i,j}\ar[rr, >->]\ar[dd,->>] & & C_{i,j+1}\ar[dd,->>] \\
        D_{i,j}\ar[ur,->>,gray]\ar[rr, >->, crossing over]\ar[dd,->>] & & D_{i,j+1}\ar[ur,->>,gray] & \\
        & C_{i+1 ,j}\ar[rr,>->] & & C_{i+1,j+1}\\
        D_{i+1,j}\ar[ur,->>,gray]\ar[rr,>->] & & D_{i+1,j+1}\ar[from=uu,->>, crossing over] &\\
      \end{tikzcd}\end{equation} where all squares are  either commutative or distinguished as appropriate. We first use condition (i) of \cref{defn:protoWald cat} to complete the span on the back face to an object $X\in\dcat C$. Since the diagram \ref{eqn2} contains the data of a map of spans, we get an induced diagram as below left where all squares are either commutative or distinguished.
      \[\begin{tikzcd}[row sep=tiny,column sep=tiny]
        & C_{i,j}\ar[rr, >->]\ar[dd,->>] & & C_{i,j+1}\ar[dd,->>] \\
        D_{i,j}\ar[ur,->>,gray]\ar[rr, >->, crossing over]\ar[dd,->>] & & D_{i,j+1}\ar[ur,->>,gray] & \\
        & C_{i+1 ,j}\ar[rr,>->] & & X\\
        D_{i+1,j}\ar[ur,->>,gray]\ar[rr,>->] & & D_{i+1,j+1}\ar[ur,->>,gray,dashed]\ar[from=uu,->>, crossing over] &\\
      \end{tikzcd} \hspace{1.5cm}
      \begin{tikzcd}[row sep=tiny,column sep=tiny]
        & C_{i,j}\ar[rr, >->]\ar[dd,->>] & & C_{i,j+1}\ar[dd,->>] \\
        C_{i,j}\ar[ur,equal,gray]\ar[rr, >->, crossing over]\ar[dd,->>] & & C_{i,j+1}\ar[ur,equal,gray] & \\
        & C_{i+1 ,j}\ar[rr,>->] & & C_{i+1,j+1}\\
        C_{i+1,j}\ar[ur,equal,gray]\ar[rr,>->] & & X\ar[ur,->>,gray,dashed,"w_{C_{i+1,j+1}}"'] \ar[from=uu,->>, crossing over]&\\
      \end{tikzcd}\] Next, we use the data on objects of condition (ii) of \cref{defn:protoWald cat} to get a diagram as above right, where again, all squares are either commutative or distinguished. The composite of these two diagrams in the vertical (gray) direction yields the required vertical map $$D_{i+1,j+1}\twoheadrightarrow X\twoheadrightarrow C_{i+1,j+1}$$ as well as the required squares.

      For the naturality of $\tau$, we must check that for every vertical natural transformation $X\twoheadrightarrow X'$ in $T^+_n\dcat C$, the resulting diagram
      \[
\begin{tikzcd}
  D_{ij}\rar[->>]\dar[->>,"w_{C_{ij}}"'] & D'_{ij}\dar[->>,"w_{C'_{ij}}"]\\
  C_{ij}\rar[->>] & C'_{ij}
\end{tikzcd}
\]
commutes in $\cat V_{\dcat C}$. This is precisely the naturality of $w$ given by condition (ii) of \cref{defn:protoWald cat}.
\end{proof}

\begin{remark}\label{rmk:sectionsnotsimplicial}
    Note that the sections $F_n$ constructed above do not necessarily assemble into a map of simplicial categories. Indeed, if we examine the action of the inner faces, having a simplicial map would require that the composite of two span completions as below left
    \[
    \begin{tikzcd}
       A\ar[rd, phantom, "\square"]\ar[r, >->] \ar[d, ->>] & B\ar[rd, phantom, "\square"]\ar[r, >->] \ar[d, ->>,dashed] & C \ar[d, ->>,dashed]\\
       A'\ar[r, >->, dashed]  & B'\ar[r, >->,dashed]  & C'
    \end{tikzcd} \hspace{2cm}
    \begin{tikzcd}
       A\ar[rrd, phantom, "\square"]\ar[r, >->] \ar[d, ->>] & B\ar[r, >->]  & C \ar[d, ->>,dashed]\\
       A'\ar[rr, >->, dashed]  &   & D
    \end{tikzcd}
    \] agree with the span completion of the composite horizontal maps as above right. However, we do not expect this condition to hold in most examples of interest; for instance, when span completions are obtained from pushouts, the objects $C'$ and $D$ above will only agree \emph{up to isomorphism}.
\end{remark}


We now construct the homotopy equivalence between $T^+_\bullet\dcat C$ and $S^{\square}_\bullet\dcat C$.

\begin{proposition}
   For any squares category $\dcat C$, the forgetful map $U\colon T^+_\bullet\dcat C\to S^{\square}_\bullet\dcat C$ is a map of simplicial categories which is a homotopy equivalence after realization.
\end{proposition}
\begin{proof}
    For each $n$, the map $U_n$ takes an object in $T^+_n\dcat C$ (that is, a diagram as in \cref{eqn1}) to the subdiagram obtained by deleting its top row, which is an object in $S^{\square}_\bullet\dcat C$. This is clearly functorial, and it is straightforward to verify that it assembles into a map of simplicial categories.

    We now define a section functor $F_n\colon S^{\square}_n\dcat C\to T^+_n\dcat C$ for each $n$. This assignment takes an object in $S_n\dcat C$ as depicted below left to the object of $T^+_n\dcat C$ depicted below right.
    
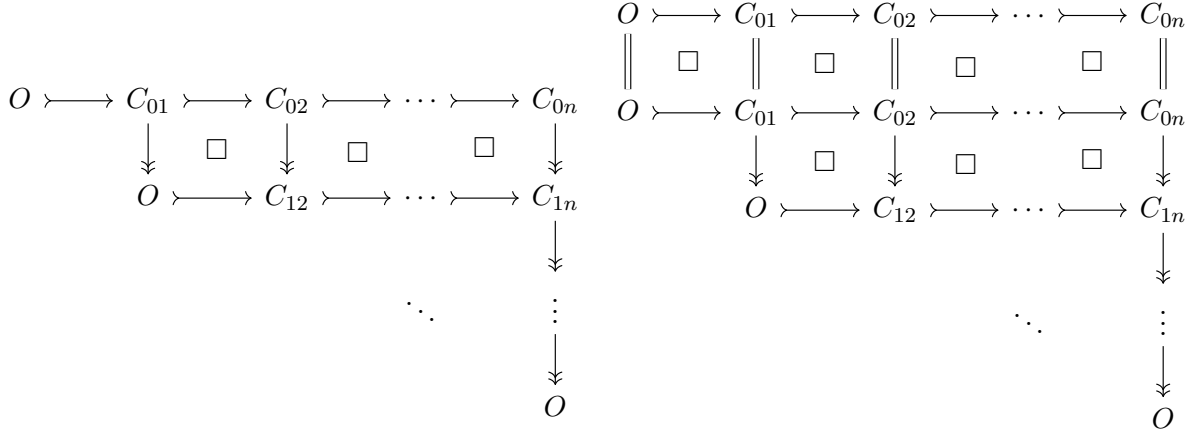
\[\begin{tikzcd}
    &\dar[phantom]&&&\\
    O \ar[r, >->]& C_{01}\ar[rd, phantom, "\square"]\ar[d, ->>] \ar[r, >->] & C_{02} \ar[r, >->] \ar[d, ->>] \ar[rd, phantom, "\square"] & \cdots \ar[r, >->] \ar[rd, phantom, "\square"]  & C_{0n} \ar[d, ->>]\\
    &O\ar[r, >->] & C_{12} \ar[r, >->] & \cdots \ar[r, >->]  & C_{1n} \ar[d, ->>]\\
    &&&  \ddots & \vdots \ar[d, ->>]\\
    &&&&  O
    \end{tikzcd} \ \
    \begin{tikzcd}
        O\ar[rd, phantom, "\square"] \ar[d, equal]\ar[r, >->] & C_{01} \ar[r, >->]\ar[d, equal] \ar[rd, phantom, "\square"] & C_{02} \ar[r, >->] \ar[d, equal] \ar[rd, phantom, "\square"] & \cdots \ar[r, >->]  \ar[rd, phantom, "\square"]  & C_{0n} \ar[d, equal] \\
    O \ar[r, >->]& C_{01}\ar[rd, phantom, "\square"]\ar[d, ->>] \ar[r, >->] & C_{02} \ar[r, >->] \ar[d, ->>] \ar[rd, phantom, "\square"] & \cdots \ar[r, >->] \ar[rd, phantom, "\square"]  & C_{0n} \ar[d, ->>]\\
    &O\ar[r, >->] & C_{12} \ar[r, >->] & \cdots \ar[r, >->]  & C_{1n} \ar[d, ->>]\\
    &&&  \ddots & \vdots \ar[d, ->>]\\
    &&&&  O
    \end{tikzcd}
    \] Clearly $F_n$ is a functor, defined on maps in the evident way. Note that these functors do not assemble into a simplicial map, as they do not commute with $d_0$. Moreover, we have $U_nF_n=\id$; we conclude our proof by constructing a natural transformation $\tau\colon\id\Rightarrow F_nU_n$.

For each object $X\in T^+_n\dcat C$, the component $\tau_X$ is the vertical natural transformation
    
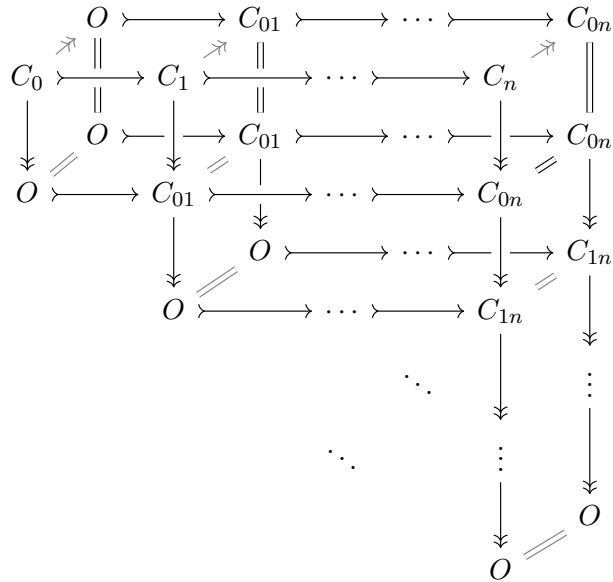
\[
\begin{tikzcd}[row sep=tiny, column sep=tiny]
                              & O \ar[rr,>->] \ar[dd, equal]&                                            & C_{01}\ar[rr, >->]\ar[dd, equal]   &                   & \cdots\ar[rr, >->] &                               & C_{0n}\ar[dd, equal]\\
C_0\ar[ur,->>,gray] \ar[rr, >->,crossing over]\ar[dd, ->>] &                 & C_1\ar[ur,->>,gray] \ar[rr, >->, crossing over]  &                               & \cdots\ar[rr, >->] &                   & C_n \ar[ur,->>,gray] &  \\
                              & O   \ar[rr,>->] &                                            & C_{01}\ar[rr, >->]\ar[dd, ->>]&                   & \cdots\ar[rr, >->] &                               & C_{0n}  \ar[dd, ->>]  \\
O\ar[ur,equal,gray]  \ar[rr, >->]              &                 & C_{01}\ar[ur,equal,gray] \ar[rr, >->,crossing over]\ar[dd, ->>]   \ar[from=uu,->>, crossing over]          &                               & \cdots\ar[rr, >->] &                   & C_{0n}\ar[ur,equal]\ar[from=uu,->>, crossing over]  & \\
                              &                 &                                            & O \ar[rr, >->]                &                   & \cdots\ar[rr, >->] &                               & C_{1n} \ar[dd, ->>] \\
                              &                 & O\ar[rr, >->] \ar[ur,equal,gray]                &                               & \cdots\ar[rr, >->] &                   & C_{1n}\ar[ur,equal,gray] \ar[dd, ->>]\ar[from=uu,->>, crossing over]&  \\
                              &                 &                                            &                               &                   &    \ddots         &                               & \vdots\ar[dd, ->>]\\
                              &                 &                                            &                               &    \ddots         &                   &  \vdots\ar[dd, ->>]           & \\
                              &                 &                                            &                               &                   &                   &                               & O\\
                              &                 &                                            &                               &                   &                   &      O  \ar[ur,equal,gray]  & \\
    \end{tikzcd}
    \] It is straightforward to verify that $\tau$ is natural.
\end{proof}

As an immediate corollary, we have the following result.

\begin{theorem}\label{thm:square and S comparison}
    If $\dcat C$ is a proto-Waldhausen category, there is an equivalence of spaces $\abs{T_\bullet\dcat C}\xrightarrow{\simeq}\abs{S^{\square}_\bullet \dcat C}$.
\end{theorem}

In particular we deduce that for any Waldhausen category the $K$-theory space produced by our $S^{\square}_\bullet$-construction, whose staircase diagrams involve pushouts up to weak equivalence, recovers the correct space up to homotopy.

\begin{corollary}\label{cor:WaldhausenisProtoWaldhausen}
    The $S^{\square}_\bullet$-construction of \cref{defn:S-dot of cat w squares} gives another model for the $K$-theory of a Waldhausen category.
\end{corollary}
\begin{proof}
  This is a consequence of \cref{ex:waldisprotowald,thm:square and S comparison}.
\end{proof}

\section{Examples}\label{sec:examples}

 In this section we include a compilation of examples of squares categories. For each of them, we determine their classes of horizontal and vertical weak equivalences, and whether they are proto-Waldhausen.  

\subsection{Finite sets}\label{ex:FinSet}
 
    The category $\Fin\Set$ of finite sets admits a squares structure where horizontal morphisms are injections, vertical morphisms are opposites of injections, and squares are underlying pushout squares. The distinguished object is the empty set.  This example can be found in \cite[Example 1.12]{CKMZ:squares}. 

    Interpreting \cref{defn:weakequiv} in this context, we see that horizontal weak equivalences are bijections and the vertical weak equivalences are opposites of bijections. This example is proto-Waldhausen since every span can be completed\[
    \begin{tikzcd}
        A \ar[r, >->] \ar[d, ->>] & B\\
        C
    \end{tikzcd} \leadsto \begin{tikzcd}
        A \ar[r, >->] \ar[d, ->>] \ar[rd, phantom, "\square"] & B\ar[d, ->>] \\
        C \ar[r, >->] & C\cup B\setminus A
    \end{tikzcd} \] and $C\cup B\setminus A$ is uniquely determined up to isomorphism. Note that these squares are stable: a commutative square of finite sets and injections is a pushout if and only if it is also a pullback.

\subsection{Polytopes}\label{ex:polytopes}
 
   Let $G$ be a subgroup of the group of isometries of $\RR^n$. There is a squares category $\dcat P^n_G$ whose objects are polytopes in $\RR^n$, where a polytope is a finite union of $n$-simplices in $\RR^n$ (see \cite[Section 2.1]{malkiewich} for a more detailed definition). The horizontal morphisms are inclusions in $\RR^n$ of the form $g\cdot P\subseteq Q$ where $g\in G$ is an isometry, and the vertical morphisms are opposites of these. The distinguished object is the empty set, and a square\[
    \begin{tikzcd}
        P \ar[r, >->, "g_0"] \ar[d, ->>, swap, "g_1^{\op}"] & Q \ar[d, ->>, "g_3^{\op}"] \\
        Q' \ar[r, >->, swap, "g_2"] & R
    \end{tikzcd}
    \] is an underlying commutative square so that $Q' = g_1^{-1}P\cap g_2^{-1}R$ and $g_0 P\cup_{g_0g_1 Q'} g_3 R = Q$. We emphasize that the intersection is taken in the category of polytopes and the union is taken as subsets of $\RR^n$, so for example, the following is a square
    \[\begin{tikzpicture}[scale=0.75]
    
    \draw[thick, fill=blue!20] (0,0) --(1,0) -- (1,1) -- (0,0);
    \draw[blue, thick, fill=blue!40] (0.5, 0.5) -- (1,0) -- (0.75, 0.75) -- (0.5, 0.5);

    \path (3,0.5) node[font=\Huge] {\(\cof\)};

    \draw[blue!60, fill=blue!20] (6,1) -- (6,0)-- (5,0) -- (6,1);
    \draw[thick] (5,1) --(6,1) -- (6,0)-- (5,0) -- (5,1);
    \draw[blue, thick, fill=blue!40] (5.5, 0.5) -- (6,0) -- (5.75, 0.75) -- (5.5, 0.5);

    \path (0.5,-1.5) node[font=\Huge, rotate=-90] {\(\quot\)};
    \path (5.5,-1.5) node[font=\Huge, rotate=-90] {\(\quot\)};

    \draw[blue, thick, fill=blue!40] (0.5,-3.5) -- (1,-4) -- (0.75, -3.25) -- (0.5, -3.5);

    \path (3,-3.5) node[font=\Huge] {\(\cof\)};
    
    \draw[thick] (5,-4) -- (6,-3) -- (5,-3) -- (5,-4);
    \draw[blue, thick, fill=blue!40] (5.5, -3.5) -- (6, -4) -- (5.75, -3.25) -- (5.5, -3.5);
    \end{tikzpicture},\] 
    even though the intersection of $P$ and $R$ as subsets of $\RR^2$ would include the entire diagonal line of the square.

      The weak equivalences in this squares category are given by the isometries in $G$: the horizontal equivalences are actual isometries and the vertical equivalences are opposites of isometries. This example is proto-Waldhausen, again essentially because we can take complements. For instance, we have\[\begin{tikzpicture}[scale=0.75]
    
    \draw[thick, fill=blue!20] (0,0) --(1,0) -- (1,1) -- (0,0);
    \draw[blue, thick, fill=blue!40] (0.5, 0.5) -- (1,0) -- (0.75, 0.75) -- (0.5, 0.5);

    \path (3,0.5) node[font=\Huge] {\(\cof\)};

    \draw[blue!60, fill=blue!20] (6,1) -- (6,0)-- (5,0) -- (6,1);
    \draw[thick] (5,1) --(6,1) -- (6,0)-- (5,0) -- (5,1);
    \draw[blue, thick, fill=blue!40] (5.5, 0.5) -- (6,0) -- (5.75, 0.75) -- (5.5, 0.5);

    \path (0.5,-1.5) node[font=\Huge, rotate=-90] {\(\quot\)};

    \draw[blue, thick, fill=blue!40] (0.5,-3.5) -- (1,-4) -- (0.75, -3.25) -- (0.5, -3.5);

 \path (7,-1.5) node[font=\Huge, rotate=0] {\(\leadsto\)};
    \end{tikzpicture}
~~~~~~~~~~~~~~~~~~~~~~~~~~~~~~~~~~~~~~~~~~~~~~~~~~~~~~~~~~~~~~~~~~~~
    \begin{tikzpicture}[scale=0.75]
    
    \draw[thick, fill=blue!20] (0,0) --(1,0) -- (1,1) -- (0,0);
    \draw[blue, thick, fill=blue!40] (0.5, 0.5) -- (1,0) -- (0.75, 0.75) -- (0.5, 0.5);

    \path (3,0.5) node[font=\Huge] {\(\cof\)};

    \draw[blue!60, fill=blue!20] (6,1) -- (6,0)-- (5,0) -- (6,1);
    \draw[thick] (5,1) --(6,1) -- (6,0)-- (5,0) -- (5,1);
    \draw[blue, thick, fill=blue!40] (5.5, 0.5) -- (6,0) -- (5.75, 0.75) -- (5.5, 0.5);

    \path (0.5,-1.5) node[font=\Huge, rotate=-90] {\(\quot\)};
    \path (5.5,-1.5) node[font=\Huge, rotate=-90] {\(\quot\)};

    \draw[blue, thick, fill=blue!40] (0.5,-3.5) -- (1,-4) -- (0.75, -3.25) -- (0.5, -3.5);

    \path (3,-3.5) node[font=\Huge] {\(\cof\)};
    
    \draw[thick] (5,-4) -- (6,-3) -- (5,-3) -- (5,-4);
    \draw[blue, thick, fill=blue!40] (5.5, -3.5) -- (6, -4) -- (5.75, -3.25) -- (5.5, -3.5);
    \end{tikzpicture}\] 
    which is a square in $\dcat P^2_G$ (we do not require polytopes to be convex). In general (suppressing the data of the specified isometries in the morphisms), a span $P \twoheadleftarrow Q\cof R$ corresponds to a sequence $P\cof Q\cof R$ and may be completed to a square via $P \cof P\cup (\overline{Q\setminus R})\cof R$; here we are using the fact that the complement of a polytope inclusion is again a polytope and the union of two polytopes is a polytope. Any other polytope $Q'$ that completes the span is necessarily isometric to $P\cup (\overline{Q\setminus R})$ (since the inclusion $Q'\cof R$ is an isometry onto its image).

\subsection{$SK$-manifolds}\label{ex:og SK}

In \cite{hoekzema/merling/murray/rovi/semikina:2021}, Hoekzema--Merling--Murray--Rovi--Semikina define a category ${\rm Mfld}_d^{\bndry}$ with squares whose $K_0^{\square}$ is an $SK$-group for manifolds with boundary; these groups encode how manifolds of a fixed dimension may be ``cut up'' and ``pasted'' back together \cite{KKNO73}. The objects of ${\rm Mfld}_d^{\bndry}$ are compact, orientable $d$-manifolds with boundary and $\cat H{\rm Mfld}_d^{\bndry} = \cat V {\rm Mfld}_d^{\bndry} = \hom({\rm Mfld}_d^{\bndry})$ are \textit{$SK$-embeddings}, which are embeddings with an additional condition on the boundary (see \cite[Definition 4.1]{hoekzema/merling/murray/rovi/semikina:2021}). The distinguished object is the empty manifold $\varnothing$ and squares are pushout squares.

In order to make ${\rm Mfld}_d^{\bndry}$ a squares category in the sense of \cref{defn:squarescat}, we need to take opposites of the vertical morphisms so that $\varnothing$ is initial in horizontal morphisms and terminal in vertical ones. As noted in \cref{rmk:O choices}, this change does not affect the resulting $K$-theory space. For clarity, we will denote this squares category $\widetilde{\rm{Mfld}}_d^{\bndry}$ to distinguish it from the original definition.

 The horizontal weak equivalences in $\widetilde{{\rm Mfld}}_d^{\bndry}$ are diffeomorphisms (rel boundary) and vertical weak equivalences are the opposites of diffeomorphisms (rel boundary). However, the squares category $\widetilde{\rm Mfld}_d^{\bndry}$ is not proto-Waldhausen because not every span of morphisms can be completed to a square. Recall that\[
\begin{tikzcd}
    N \ar[r, >->] \ar[d, ->>] & M\\
    M'
\end{tikzcd} \text{ in }\widetilde{\rm Mfld}_d^{\bndry} ~\leadsto \begin{tikzcd}
    N \ar[r, hook] & M\\
    M' \ar[u, hook]
\end{tikzcd} \text{ in }{\rm Mfld}_d^{\bndry}
\] where $\hookrightarrow$ denotes the $SK$-embeddings from \cite[Definition 4.1]{hoekzema/merling/murray/rovi/semikina:2021}. In particular, taking $M'=\varnothing$, completing the span above to a square is like asking for $N\hookrightarrow M$ to have a complement. Although the conditions on $SK$-embeddings ensure that $\overline{M\setminus N}$ is an object of ${\rm Mfld}_d^{\bndry}$, the commuting square\[
\begin{tikzcd}
    N \ar[r, hook] & M\\
    \varnothing \ar[u, hook] \ar[r, hook] & \overline{M\setminus N} \ar[u, hook]
\end{tikzcd}
\]is not a pushout in ${\rm Mfld}_d^{\bndry}$. 

There are two potential approaches one could take in order to solve this shortcoming. One option is to give an $S_\bullet$-construction for ${\rm Mfld}^d_{\bndry}$ following \cref{rmk:augmented Sdot}, with the augmentation $\cat A$ given by cylinders $N\times I$ for $N$ a closed, orientable $(d-1)$-manifold. One could define a notion of augmented proto-Waldhausen category and make analogous arguments, although we do not pursue these ideas in this paper. Another option is to adjust the definition of ${\rm Mfld}_d^{\bndry}$ so that the commutative squares above are distinguished. One could define a notion of \textit{$SK$-equivalence} and then ask for squares to be ``pushouts up to $SK$-equivalences,'' much like the Thomason construction. We intend to further explore these ideas, together with the resulting $K$-theory, in future work. 

\subsection{Partial monoids}\label{ex:partial monoid}

The next example is \cite[Example 7.1]{BOORS:2016}. Recall that a \textit{partial monoid} is a set $M$ with a \textit{partial operation} $*\colon M_2\to M$ for some $M_2\subseteq M\times M$. This partial operation is required to have a unit $1\in M$ and be associative when defined (c.f. \cite[Example 2.1]{BOORS:2016}). The \textit{nerve} of a partial monoid is the simplicial set $N_\bullet M$ with $N_0M=\{1\}$, $N_1M = M$, and for $k\geq 1$, $N_k M\subseteq M^{\times k}$ are the composable $k$-tuples, i.e. those $(m_1,\dots, m_k)$ so that $(m_1*\cdots*m_i, m_{i+1})\in M_2$ for all $i$. The face maps apply the operation $*$ in the appropriate slot and the degeneracies insert the unit. The \textit{classifying space}  $BM$ is the realization of this simplicial set.

There is a squares category $\dcat M$ whose $K^{\square}$-theory space is $\Omega BM$. The objects of $\dcat M$ are the elements $m\in M$ and the distinguished object is the unit. Horizontal and vertical morphisms are witnessed by right- and left-multiplication, respectively\[
\begin{tikzcd}
    a \ar[rr, >->, "{(a,b)\in M_2}"] && a*b
\end{tikzcd} \text{ and } \begin{tikzcd}
    a*b \ar[rr, ->>, "{(a,b)\in M_2}"] && b
\end{tikzcd} 
\] and squares witness the associativity of $*$,\[
\begin{tikzcd}
    c*a \ar[rr, >->, "{(c*a,b)\in M_2}"] \ar[dd, swap, ->>, "{(c,a)\in M_2}"]  \ar[rrdd, phantom, "\square"] && c*a*b \ar[dd, ->>, "{(c,a*b)\in M_2}"] \\
    &&\\
    a \ar[rr, >->, swap, "{(a,b)\in M_2}"]  && a*b 
\end{tikzcd}.
\] To see that $K^{\square}(\dcat M)\cong \Omega B M$, observe that the double nerve $\dcat{N}\dcat M$ is isomorphic as a simplicial set to the edgewise subdivision of $N_\bullet M$. Hence \[
\abs{T_\bullet \dcat M}\cong \abs{\dcat{N}\dcat{M}} \cong \abs{sd(N_\bullet M)}\cong B M
\] using the fact that the edgewise subsivision of a simplicial set does not change its geometric realization \cite{segal:1973}.

The horizontal and vertical weak equivalences in this squares category are simply the identities, and it is straightforward to check that this example is proto-Waldhausen. In fact, this is an example of a \textit{stable pointed double category} (a definition from \cite{BOORS:2016} which we recall later in \cref{defn:stable dbl cat}) which is a strictly stronger notion.

\subsection{Graphs}\label{ex:graphs}
   There are several ways to obtain a squares category from graphs:
   \begin{itemize}
    \item[(1)] The double category of graphs described in \cite[Example 7.3]{BOORS:2016} is a squares category, where the distinguished object is the empty graph. For a fixed ambient graph $G$, the objects of this double category are subgraphs $H\hookrightarrow G$, horizontal morphisms are full subgraph inclusions (rel $G$) and vertical morphisms are opposites of full subgraph inclusions (rel $G$). Note that a full subgraph inclusion $H'\hookrightarrow H$ is equivalently specified by a partition $V(H) = V(H') \amalg V(H')^c$. Squares are as described in \cite[Example 7.3]{BOORS:2016}, specified by a partition of vertices into three pieces; for instance, \[
    \begin{tikzpicture}[scale=0.75]
    \fill[blue] (0,0) circle (3pt);
    \fill (1,1) circle (3pt);
    \fill (0,1) circle (3pt);
    
    \draw (0.2,1) --(0.8,1);
    \draw (0,0.2) --(0,0.8);

    \path (3,0.5) node[font=\Huge] {\(\cof\)};

    \fill[blue] (5,0) circle (3pt);
    \fill (6,1) circle (3pt);
    \fill (5,1) circle (3pt);
    \fill (6,0) circle (3pt);
    
    \draw (5.2,1) --(5.8,1);
    \draw (5,0.2) --(5,0.8);
    \draw (5.2,0) --(5.8,0);
    \draw (6,0.2) --(6,0.8);

    \path (0.5,-1.5) node[font=\Huge, rotate=-90] {\(\quot\)};
    \path (5.5,-1.5) node[font=\Huge, rotate=-90] {\(\quot\)};

    \fill[blue] (0,-3) circle (3pt);

    \path (3,-3) node[font=\Huge] {\(\cof\)};

    \fill[blue] (5,-3) circle (3pt);
    \fill (6,-3) circle (3pt);
    
    \draw (5.2,-3) --(5.8,-3);
    \end{tikzpicture}
    \] is a square and the corresponding partition of vertices is\[
    \begin{tikzpicture}[scale=1.5]
    \fill (0,0) circle (3pt);
    \fill (1,1) circle (3pt);
    \fill (0,1) circle (3pt);
    \fill (1,0) circle (3pt);
    
    \draw[thick] (0.2,1) --(0.8,1);
    \draw[thick] (0,0.2) --(0,0.8);
    \draw[thick] (0.2,0) --(0.8,0);
    \draw[thick] (1,0.2) --(1,0.8);

    \draw[blue, thick] (0,0) circle (8pt);
    \node at (-0.5,0) {\textcolor{blue}{$1$}};
    
    \draw[green!50!black, thick] (1,0) circle (8pt);
    \node at (1.5,0) {\textcolor{green!50!black}{$3$}};
    
    \draw[orange, thick] (0.5, 1) ellipse (20pt and 10pt);
    \node at (0.5,1.5) {\textcolor{orange}{$2$}};
    \end{tikzpicture}.
    \]   
        \item[(2)] Horizontal morphisms are subgraph inclusions, vertical morphisms are opposites of subgraph inclusions, and squares are underlying commutative diagrams which are pushouts on the sets of vertices. The difference is that we are not working relative to a fixed graph $G$.
        \item[(3)] Horizontal morphisms are subgraph inclusions, vertical morphisms are opposites of subgraph inclusions, and squares are underlying pushouts in the category of finite graphs. The $K^{\square}$-theory of this squares category should essentially be the $K$-theory of the category with covering families from \cite[Definition 3.6]{Calle/Gould:24}. However, there is no comparison between square $K$-theory and the $K$-theory for categories with covering families at this time, although it is likely one could pursue such a comparison using ideas of \cite[Section 1.8]{waldhausen:1983}.
    \end{itemize}

    The weak equivalences in the squares category of graphs of (1) are simply the identities; hence this example is proto-Waldhausen.
    For (2) and  (3), the weak equivalences are given by graph isomorphisms. Example (2) is proto-Waldhausen, where the span completion $s$ takes \[
\begin{tikzcd}
    H \ar[r, >->] \ar[d, ->>] & G\\
    H'
\end{tikzcd} \leadsto \begin{tikzcd}
    H \ar[r, >->] \ar[d, ->>] & G \ar[d, ->>] \\
    H' \ar[r, >->] & G'
\end{tikzcd}
\] where $G'$ is the full subgraph of $G$ on $V(H')\cup V(G)\setminus V(H)$; the components of the natural transformation $w$ are graph inclusions which are identity on objects.
Example (3) is \textit{not} proto-Waldhausen since not every span can be completed. For instance, for the specific example illustrated previously, there is no way to complete the span to a square which is a pushout in the category of graphs.







\subsection{Proto-exact categories}\label{ex:protoexact}

Proto-exact categories were introduced in \cite[\S 2.4]{dyckerhoff/kapranov:19} as a generalization of exact categories. Just as Waldhausen categories generalize exact categories, every proto-exact category is a proto-Waldhausen category in a natural way.

\begin{definition}
    A \textit{proto-exact category}  is a pointed category $\cat C$ with two distinguished classes of morphisms $\cat M$ and $\cat E$ called \textit{admissible monomorphisms} and \textit{admissible epimorphisms}, respectively, satisfying the following:\begin{itemize}
        \item Augmented: the zero object $*\in \cat C$ is initial in $\cat M$ and terminal in $\cat E$. Any morphism $*\to A$ is in $\cat M$, and any morphism $A\to *$ is in $\cat E$.
        \item Closure: $\cat M$ and $\cat E$ are closed under composition and contain all isomorphisms.
        \item Bicartesian squares: a commutative square of the form\[
        \begin{tikzcd}
            \scriptstyle\bullet \ar[r, >->] \ar[d, ->>] & \scriptstyle\bullet \ar[d, ->>] \\
            \scriptstyle\bullet \ar[r, >->] & \scriptstyle\bullet
        \end{tikzcd}
        \] is cartesian if and only if it is cocartesian. We distinguish bicartesian squares with a $\square$ in the center.
        \item Stable: Every span and cospan determine a bicartesian square,\[
        \begin{tikzcd}
            \scriptstyle\bullet \ar[r, >->] \ar[d, ->>] & \scriptstyle\bullet  \\
            \scriptstyle\bullet  &
        \end{tikzcd}\mapsto\begin{tikzcd}
            \scriptstyle\bullet \ar[r, >->] \ar[d, ->>] \ar[rd, phantom, "\square"] & \scriptstyle\bullet \ar[d, ->>] \\
            \scriptstyle\bullet \ar[r, >->] & \scriptstyle\bullet
        \end{tikzcd} ~\text{ and }~ \begin{tikzcd}
             & \scriptstyle\bullet \ar[d, ->>] \\
            \scriptstyle\bullet \ar[r, >->] & \scriptstyle\bullet
        \end{tikzcd}\mapsto\begin{tikzcd}
            \scriptstyle\bullet \ar[r, >->] \ar[d, ->>] \ar[rd, phantom, "\square"] & \scriptstyle\bullet \ar[d, ->>] \\
            \scriptstyle\bullet \ar[r, >->] & \scriptstyle\bullet
        \end{tikzcd}
        \]
    \end{itemize}
\end{definition}

Examples of proto-exact categories include the category of finite pointed sets and the category of finitely generated projective $R$-modules.

A proto-exact category determines a category with squares in a natural way, whose horizontal and vertical weak equivalences are the isomorphisms in $\cat M$ and $\cat E$, respectively. Moreover, using the universal properties of these bicartesian squares, one can check that these squares categories are proto-Waldhausen. Note that the fact that squares are bicartesian also implies that both horizontal and vertical weak equivalences are isomorphisms. Then by \cref{thm:square and S comparison}, we have $\abs{S^{\square}_\bullet\cat C}\simeq \abs{T_\bullet \cat C}$ as spaces.

\subsection{Weak Waldhausen categories}\label{ex:weak Waldhausen}

    In \cite{ogawa/shah}, Ogawa--Shah introduce the notion of a \textit{weak Waldhausen category}, which is a generalization of a Waldhausen category that also includes triangulated categories as examples, and define a Grothendieck group for such categories. Every weak Waldhausen category $\cat C$ defines a squares category, where the horizontal morphisms are the cofibrations, the vertical morphisms are all morphisms in $\cat C$, and the squares are the weak pushout squares of \cite[Remark 2.14(3)]{ogawa/shah}. Morphisms of weak Waldhausen categories, defined in \cite[Definition 2.18]{ogawa/shah}, are precisely the weakly exact functors described in \cref{rmk:weaklyexact}. In particular, the $S^{\square}_\bullet$-construction is functorial in weak Waldhausen categories and weakly exact functors, and produces a space $K(\cat C):= \Omega \abs{S^{\square}_\bullet \cat C}$ so that $K_0(\cat C)$ is precisely the Grothendieck group defined by Ogawa--Shah.
    
    Weak Waldhausen categories are not proto-Waldhausen, although they are very close. In particular, even if one could define $s$ on objects, this section is only guaranteed to be defined on those morphisms in $\Fun^v(\tinysquare, \dcat C)$ which are pointwise weak equivalences (using \cite[Definition 2.13(WW1)]{ogawa/shah}). We emphasize that this entire structure is needed for the comparison $S^{\square}_\bullet$-construction and $T_\bullet$-construction. However, it is possible that certain examples weak Waldhausen categories may admit a proto-Waldhausen structure. 

\section{Connection with 2-Segal objects}\label{sec:segal connection}

A $2$-Segal set is a simplicial set $X$ which behaves like a ``multi-valued category'' in the sense that it has objects $X_0$ and morphisms $X_1$, but no well-defined notion of composition; in particular, the first map in the span \[
X_1\times_{X_0} X_1 \leftarrow X_2 \to X_1
\] need not be a bijection. However, there is a well-defined ``composition'' of $2$-simplices (whence the ``$2$'' in $2$-Segal) in the sense that the first map in the span\[
X_2\times_{X_1} X_2 \leftarrow X_3 \to X_2
\] is invertible. Moreover, this composition is associative, as witnessed by the fact that the \textit{$2$-Segal maps}\[
X_n \xrightarrow{\cong} X_2\times_{X_1}\dots \times_{X_1} X_2
\] which land in the $(n-1)$-fold iterated pullback are isomorphisms for $n\geq 3$. More generally, one has the notion of \textit{$2$-Segal spaces} which are simplicial spaces so that the $2$-Segal maps above are weak equivalences.

Since their introduction by Dyckerhoff--Kapronov in \cite{dyckerhoff/kapranov:19} and independently by G\'alvez-Carriollo--Kock--Tonks in \cite{GCKT:18}, $2$-Segal spaces have been connected to a variety of different areas of study, including the theory of Hall algebras \cite{Dyckerhoff:18, Penney:17, Walde, Young:18}, ennumerative combinatorics \cite{Carlier, Carlier-Kock, GCKT:18}, higher category theory \cite{feller, Stern:19}, and higher algebraic $K$-theory \cite{BOORS:2016, carawan, poguntke}. In \cite{dyckerhoff/kapranov:19} and \cite{GCKT:18}, both teams of authors observed a particularly striking connection between $2$-Segal sets and higher algebraic $K$-theory; namely, that Waldhausen's $S_\bullet$-construction outputs $2$-Segal objects when fed categorical inputs with enough structure.

In \cite{BOORS:2016}, Bergner--Osorno--Ozornova--Rovelli--Scheimbauer take this idea a step further and identify the precise categorical structure necessary for a version of the $S_\bullet$-construction to produce a $2$-Segal set. The structures they consider, called \textit{stable augmented double categories}, are double categories which (among other conditions) satisfy a ``stability'' condition: every span and cospan can be uniquely completed to a square\[
\begin{tikzcd}
    A \ar[r, >->] \ar[d, ->>] & B\\
    C &
\end{tikzcd} \leadsto \begin{tikzcd}
    A \ar[r, >->] \ar[d, ->>]\ar[rd, phantom, "\square"] & B\ar[d, ->>] \\
    C \ar[r, >->] & D
\end{tikzcd} \hspace{1cm} \text{ and } \hspace{1cm} \begin{tikzcd}
     & B \ar[d, ->>] \\
    C \ar[r, >->] & D
\end{tikzcd} \leadsto \begin{tikzcd}
    A \ar[r, >->] \ar[d, ->>] \ar[rd, phantom, "\square"] & B\ar[d, ->>] \\
    C \ar[r, >->] & D
\end{tikzcd}.
\] The $S_\bullet$-construction of such a double category consists of diagrams of the form\[
\begin{tikzcd}
    A_{00} \ar[r, >->]& A_{01}\ar[rd, phantom, "\square"]\ar[d, ->>] \ar[r, >->] & A_{02} \ar[r, >->] \ar[d, ->>] \ar[rd, phantom, "\square"] & \cdots \ar[r, >->] \ar[rd, phantom, "\square"]  & A_{0n} \ar[d, ->>]\\
    &A_{11} \ar[r, >->] & A_{12} \ar[r, >->] & \cdots \ar[r, >->]  & A_{1n} \ar[d, ->>]\\
    &&&  \ddots & \vdots \ar[d, ->>]\\
    &&&&  A_{nn}
\end{tikzcd}
\] where each $A_{ii}$ is in the \textit{augmentation} of the double category. The main result of \cite{BOORS:2016} is that their $S_\bullet$-construction gives an equivalence of categories between stable augmented double categories and $2$-Segal sets; the inverse is very explicit and uses a path space construction (see \cite[Section 5]{BOORS:2016}). Hence the $S_\bullet$-construction of a stable augmented double category is completely characterized by having a $2$-Segal structure.

Inspired by this result, one could ask when the $K$-theory of a category with squares determines a $2$-Segal object. In this section, we introduce \textit{stable squares categories} (\cref{defn:stable squares cat}), which are like the stable (pointed) double categories above but with weak equivalences. Because of these weak equivalences, stable squares categories do not always produce $2$-Segal objects. However, when the weak equivalences are invertible (along with one additional, mild condition, see \cref{defn:isostable}), the $S^{\square}_\bullet$-construction of a stable squares category is a $2$-Segal space (\cref{thm:2Segal from squares}). Examples that satisfy these conditions include finite sets (\cref{ex:FinSet}), polytopes (\cref{ex:polytopes}), and the Waldhausen category of finitely-generated projective $R$-modules.

\begin{remark}
    In \cite{BOORS:2021}, Bergner--Osorno--Ozornova--Rovelli--Scheimbauer broaden the scope of their results to include more homotopical contexts, using a sufficiently nice model category $\cat A$. Their main result is that a kind of $S_\bullet$-construction induces a Quillen equivalence between the category of (unital) $2$-Segal objects and a category of \textit{augmented stable double Segal objects} in $\cat A$. These augmented stable double Segal objects can be interpreted as the internal double nerve of a double $\cat A$-category (a category internal to categories internal to $\cat A$), and one could ask about connections to categories with squares and a $T_\bullet$-construction in this context. This would require developing a theory of internal squares categories, which we do not pursue in this paper (but is considered for spaces in \cite{HRS:22}).
\end{remark}

\subsection{2-Segal objects}
There are many equivalent ways to formulate what it means to be a $2$-Segal object, one of which is the following. We refer the reader to \cite{dyckerhoff/kapranov:19, BOORS:2021} for a more detailed discusion of $2$-Segal objects.

\begin{definition}\label{defn:2Segal}
    A \textit{2-Segal object in a category $\cat A$} is a functor $X\colon \Delta^{\op}\to \cat A$ so that for every $n\geq 3$, and any $0\leq i<j\leq n$, the diagram\[
\begin{tikzcd}
    \{0,\dots, n\} & \{i,\dots, j\}\ar[l] \\
    \{0,\dots, i,j,\dots,n\} \ar[u] & \{i,j\} \ar[l] \ar[u]
\end{tikzcd}
\] in $\Delta$ is sent to a homotopy pullback\[
    \begin{tikzcd}
        X_n \ar[r] \ar[d] & X_{\{i, \dots, j\}} \ar[d] \\
         X_{\{0,\dots, i,j, \dots, n\}}\ar[r] & X_{\{i,j\}}
    \end{tikzcd}
    \] in $\cat A$. We let $2\Seg_{\cat A}$ denote the category of $2$-Segal objects in $\cat A$ and maps between them (i.e. natural transformations), only using the subscript $\cat A$ when the ambient category is not clear from context.
\end{definition}

\begin{remark}\label{rmk:check 2Segal}
    As shown in \cite[Proposition 1.17]{BOORS:2021}, for each $n\geq 3$, it suffices to consider the diagrams for $(i,j)=(0,2)$ and $(i,j)=(n-2,n)$.
\end{remark}

In the literature, many authors impose the additional condition that their $2$-Segal objects are \textit{unital}, meaning that the squares\[
\begin{tikzcd}
    X_1 \ar[r, "d_0"] \ar[d, swap, "s_1"] & X_0 \ar[d, "s_0"] \\
    X_2 \ar[r, swap, "d_0"] & X_1
\end{tikzcd} ~\text{ and }~ \begin{tikzcd}
    X_1 \ar[r, "d_1"] \ar[d, swap, "s_0"] & X_0 \ar[d, "s_0"] \\
    X_2 \ar[r, swap, "d_2"] & X_1
\end{tikzcd}
\] are pullbacks (c.f. \cite[Lemma 1.11]{BOORS:2016}). However, \cite{felleretal} has since shown that every $2$-Segal space is unital, and so we will not make such a restriction.

\begin{definition}
    A $2$-Segal object is \textit{reduced} if $X(0)=*$. We let $2\Seg_*\subseteq 2\Seg$ denote the subcategory of reduced $2$-Segal objects.
\end{definition}

One well-studied example of $2$-Segal objects comes from Waldhausen's $S_\bullet$-construction on an exact category. As discussed in \cite[Remark 7.3.7]{dyckerhoff/kapranov:19}, it is known that $S_\bullet(\cat C)$ may not be $2$-Segal for an arbitrary Waldhausen category $\cat C$ (see also \cite{carawan}). However, in \cite{BOORS:2016}, Bergner--Osorno--Ozornova--Rovelli--Scheimbauer show that every $2$-Segal set arises as the $S_\bullet$-construction of a certain kind of double category, called \textit{stable augmented double categories}. For simplicity, we will restrict from augmented to pointed double categories (corresponding to reduced $2$-Segal sets), although one could similarly study the augmented case (c.f. \cref{rmk:augmented Sdot}).

\begin{definition}\label{defn:stable dbl cat}
    A double category $\mathbb{C}$ is \textit{stable} if every square is uniquely determined by its span and cospan, meaning there are bijections\[
    \cat H_{\dcat C}\times_{\ob \dcat C} \cat V_{\dcat C} \xleftarrow{}{\Sq}(\dcat C) \xrightarrow{} \cat V_{\dcat C} \times_{\ob \dcat C}  \cat H_{\dcat C}  \] given by the maps
    \[
    \begin{tikzcd}
            & \scriptstyle\bullet \ar[d, ->>] \\
            \scriptstyle\bullet\ar[r, >->] & \scriptstyle\bullet
    \end{tikzcd}
    \mapsfrom
    \begin{tikzcd}
         \scriptstyle\bullet \ar[r, >->] \ar[d, ->>]& \scriptstyle\bullet \ar[d, ->>]\\
        \scriptstyle\bullet \ar[r, >->]  & \scriptstyle\bullet
    \end{tikzcd} \mapsto
        \begin{tikzcd}
         \scriptstyle\bullet \ar[r, >->] \ar[d, ->>]& \scriptstyle\bullet\\
        \scriptstyle\bullet
    \end{tikzcd} .
    \]
\end{definition}

A stable double category $\dcat{C}$ is, by definition, a flat double category. Hence a pointed stable double category is a squares category, so we may consider its $S_\bullet^{\square}$-construction.

\begin{proposition}\label{rmk:WITissquares}
If $\dcat C$ is a pointed stable double category, then the simplicial set of objects $\ob S^{\square}_\bullet \dcat{C}$ agrees with the $S_\bullet$-construction of \cite{BOORS:2016}.
\end{proposition}
\begin{proof}
    It is straightforward to check that the definition of $S_\bullet$ given in \cite[Section 4.7]{BOORS:2016} is the simplicial set of objects of the simplicial category described in \cref{defn:S-dot of cat w squares}. 
\end{proof}

We can use this observation and the results of \cite{BOORS:2016} to immediately deduce the following.

\begin{corollary}
    Every reduced $2$-Segal set is equivalent to $\ob S^{\square}_\bullet \dcat C$ for some squares category $\dcat C$.
\end{corollary}

\begin{remark}\label{rmk:wes in st ptd dbl}
    Note that stability of $\dcat C$ implies that the maps in $w^h\dcat C$ and $w^v\dcat C$ must be \textit{equalities} (not even just isomorphisms). In particular, the same argument from \cref{prop:S-bullet with isos} implies that $\abs{S^{\square}_\bullet \dcat{C}}\simeq \abs{\ob S^{\square}_\bullet \dcat{C}}$; hence the $K^\square$-theory space as a squares category and the space obtained from the $S_\bullet$-construction of \cite{BOORS:2016} agree by \cref{thm:square and S comparison}.
\end{remark}

\subsection{Stable squares categories} We now introduce a notion of ``stability'' for squares categories, that will allow us to connect these structures to 2-Segal objects via the $S^{\square}_\bullet$-construction of \cref{section:Sdot}. The notion of a stable squares category is inspired by that of proto-exact categories; see \cref{ex:protoexact}. 
Essentially, a stable squares category satisfies both the proto-Waldhausen conditions and their duals.

\begin{definition}\label{defn:stable squares cat}
    A squares category $\dcat C$ is \textit{stable} if
    \begin{itemize}
      \item[(i)]
        The functor $i^*\colon \Fun^v\left(\tinysquare, \dcat C\right)\to \Fun^v\left(\tinyspan, \dcat C\right)$ induced by the inclusion $i\colon\tinyspan\hookrightarrow\tinysquare$ admits a section functor $s$.
        \item[(ii)] There is a natural transformation $w\colon si^*\Rightarrow \id$ whose component on a square is\[
\begin{tikzcd}[row sep=tiny, column sep=tiny]
  & A\ar[rr,>->]\ar[dd,->>] && B\ar[dd,->>]\\
  A\ar[ur,equal, gray]\ar[rr,>->,crossing over]\ar[dd,->>] && B\ar[ur,equal, gray]&\\
  & C\ar[rr,>->] && D\\
   C\ar[ur,equal, gray]\ar[rr,>->] && D_0\ar[ur,->>,gray,"w_D"']\ar[from=uu,->>, crossing over] &\\
\end{tikzcd} \hspace{1cm}
\begin{tikzcd}
  & \id\\
  si^*\ar[ur,->>,gray,"w"]\\
\end{tikzcd}.
\]
\item[(iii)] The functor  $j^*\colon \Fun^v\left(\tinysquare, \dcat C\right)\to \Fun^v \left(\tinycospan, \dcat C\right)$ induced by the inclusion $j\colon\tinycospan\hookrightarrow\tinysquare$ admits a section functor $t$.
        \item[(iv)] There is a natural transformation $u\colon tj^*\Rightarrow \id$ whose component on a square is\[
\begin{tikzcd}[row sep=tiny, column sep=tiny]
  & A\ar[rr,>->]\ar[dd,->>] && B\ar[dd,->>]\\
  A_0\ar[ur,->>,gray,"u_A"]\ar[rr,>->,crossing over]\ar[dd,->>] && B\ar[ur,equal, gray] &\\
  & C\ar[rr,>->] && D\\
   C\ar[ur,equal, gray]\ar[rr,>->] && D\ar[ur,equal,gray]\ar[from=uu,->>, crossing over] &\\
\end{tikzcd} \hspace{1cm}
\begin{tikzcd}
  & \id\\
  tj^*\ar[ur,->>,gray,"u"]\\
\end{tikzcd}.
\] and such that the map $u_A$ is a weak equivalence for all $A$.
    \end{itemize}
\end{definition}

\begin{remark}
    The first two conditions in the above definition just say that $\dcat C$ is proto-Waldhausen. Conditions (iii) and (iv) are analogous to (i) and (ii) but for cospans instead of spans, the only notable difference being that we need the separate requirement that the natural transformation is point-wise a weak equivalence, as this is no longer implied by the rest of the data. Dually to the proto-Waldhausen axioms, conditions (iii) and (iv) invoke the idea that distinguished squares are pullbacks up to weak equivalence. In \cref{rmk:stablealternatives}, we explain our choice of ``dualization'' that only includes vertical natural transformations, and mention other plausible notions of stability. 
\end{remark}

Our first example of stable squares category shows how they provide a generalization of the pointed stable double categories of \cite{BOORS:2016}.

\begin{proposition}\label{prop:stable dblcat is stable sqcat}
    Every pointed stable double category is a stable squares category.
\end{proposition}\begin{proof}
As explained in \cref{rmk:WITissquares}, a pointed stable double category is a squares category in which the weak equivalences $w^h\dcat C$ and $w^v\dcat C$ are just the identity morphisms. For the stability conditions, recall that every span and cospan completes uniquely to a square; this defines the functors $s$ and $t$ on objects. To define the functor $s$ on morphisms, suppose we have a diagram as below left
\[\begin{tikzcd}[row sep=tiny, column sep=tiny]
  & A'\ar[rr,>->]\ar[dd,->>] && B'\ar[dd,->>]\\
  A\ar[ur,->>]\ar[rr,>->,crossing over]\ar[dd,->>] && B\ar[ur,->>] &\\
  & C'\ar[rr,>->] && D'\\
   C\ar[ur,->>]\ar[rr,>->] && D\ar[from=uu,->>, crossing over] &\\
\end{tikzcd} \hspace{2cm}
\begin{tikzcd}
    C\dar[->>]\rar[>->]\ar[dr,phantom, "\square"] & D\dar[->>]\\
    C'\rar[>->] & E
\end{tikzcd}
\] where all squares are  either commutative or distinguished as appropriate. Completing the span 
\[C' \twoheadleftarrow C\cof D\] yields a square as depicted above right. We now have two squares given by the vertical composites
\[\begin{tikzcd}
A\dar[->>]\rar[>->]\ar[dr,phantom, "\square"] & B\dar[->>]\\
    A'\dar[->>]\rar[>->]\ar[dr,phantom, "\square"] & B'\dar[->>]\\
    C'\rar[>->] & D'
\end{tikzcd} \hspace{2cm}
\begin{tikzcd}
A\dar[->>]\rar[>->]\ar[dr,phantom, "\square"] & B\dar[->>]\\
    C\dar[->>]\rar[>->]\ar[dr,phantom, "\square"] & D\dar[->>]\\
    C'\rar[>->] & E
\end{tikzcd}
\] both of which complete the same span
\[C' \twoheadleftarrow A\cof B.\] By assumption, these must agree; then our required map $D\twoheadrightarrow D'$ is the map $D\twoheadrightarrow E$ and the two new squares created in the cube are  either commutative or distinguished. 

Since $s$ was constructed using the unique span completions, it must be functorial. One defines the functor $t$ analogously, using the unique cospan completions. Finally, the required natural transformations $w$ and $u$ are simply the identity.
\end{proof}

\begin{example}
By \cite{BOORS:2016}, every pointed $2$-Segal set $X$ determines a pointed stable double category, and hence a stable squares category. In the case that $X = S_\bullet \dcat C$ is the $S_\bullet$-construction of \cite{BOORS:2016} on a pointed stable double category $\dcat C$, we detail the structure of the corresponding category of squares $\dcat X$. The objects of $\dcat X$ are $*\cof c\in S_1\dcat C$; horizontal morphisms $(*\to c_0) \cof (* \to c_1)$ and  vertical morphisms $(*\to d_0) \quot (*\to d_1)$ are elements of $S_2\dcat C$ of the form \[
    \begin{tikzcd}
        * \ar[r,>->] & c_0 \ar[r, >->] \ar[d, ->>] \ar[rd, phantom, "\square"] & c_1\ar[d, ->>]\\
        & * \ar[r, >->] & c_{01} \ar[d, ->>] \\
        && *
    \end{tikzcd} ~\text{ and }~
    \begin{tikzcd}
        * \ar[r,>->] & c_0 \ar[r, >->] \ar[d, ->>]  \ar[rd, phantom, "\square"]& d_0\ar[d, ->>]\\
        & * \ar[r, >->] & d_1 \ar[d, ->>] \\
        && *
    \end{tikzcd},
    \] respectively. A distinguished square in $\dcat X$ is an element of $S_3\dcat C$,\[
    \begin{tikzcd}
(*\to c_0) \ar[r, >->] \ar[d, ->>]  \ar[dr, phantom, "\square"]& (*\to c_1) \ar[d, ->>]\\
(*\to d_0) \ar[r, >->]& (*\to d_1)
\end{tikzcd} ~=~ \begin{tikzcd}
    * \ar[r, >->] & c \ar[r, >->] \ar[d, ->>] \ar[rd, phantom, "\square"] & c_0 \ar[r, >->] \ar[d, ->>]  \ar[rd, phantom, "\square"]& c_1 \ar[d, ->>]\\
      & * \ar[r, >->] & d_0\ar[r, >->] \ar[d, ->>]  \ar[rd, phantom, "\square"]& d_1 \ar[d, ->>]\\
      &   & * \ar[r, >->]& d \ar[d, ->>]\\
    &&& *
\end{tikzcd}
    \] and the distinguished object is the identity on $*$.
\end{example}

\begin{example} 
The squares category given by any proto-exact category (see \cref{ex:protoexact}) is stable, as one can readily check using the universal properties of bicartesian squares. 
\end{example}

\begin{example}\label{ex:FinSet stable}
    The squares category of finite sets from \cref{ex:FinSet} is stable. Because of the symmetry of the squares, the cospan completion $t$ may be defined just as the span completion $s$ was.
\end{example}

\begin{example}\label{ex:polytope stable}
    The polytope example from \cref{ex:polytopes} is stable. Similarly to finite sets, stability follows from the definition of squares in $\dcat{P}^n_G$ and their symmetry.
\end{example}

\begin{remark}
    In \cite[Section 2.3]{FL:1991}, Fiedorowicz--Loday consider double categories that satisfy a condition very similar to \cref{defn:stable squares cat}(iii) and show that every crossed simplicial group determines such a double category. In \cite[Proposition 2.6]{FL:1991}, they further show that the classifying space of these double categories can be modeled by a generalized $Q$-construction. It would be interesting to consider when the $K^{\square}$-theory of a squares category can be modeled by a similar $Q$-construction, but this is beyond the scope of this paper.
\end{remark}

\subsection{2-Segal objects from stable squares categories} 
The $S^{\square}_\bullet$-construction from \cref{defn:S-dot of cat w squares} is precisely the one defined by Dyckerhoff--Kapranov for proto-exact categories. In particular, $S^{\square}_\bullet (\cat C)$ is a simplicial groupoid when $\cat C$ is proto-exact, and in \cite[Proposition 2.4.8]{dyckerhoff/kapranov:19} Dyckerhoff--Kapranov use this to show that $[n]\mapsto BS_n(\cat C)$ is a $2$-Segal space. We can generalize their methods to stable squares categories.

\begin{definition}\label{defn:isostable}
    A squares category $\dcat C$ is \textit{isostable} if it is a stable squares category whose weak equivalences are invertible, and moreover given
      \[\begin{tikzcd}
     A\ar[d, ->>, "g"', "\cong"] \ar[r, >->,"f"]& B \ar[d, ->>,"h","\cong"'] \\
    C \ar[r, >->,"k"] & D
\end{tikzcd}
\hspace{2cm}
\begin{tikzcd}
     C\ar[d, ->>, "g^{-1}"', "\cong"] \ar[r, >->,"k"]& D \ar[d, ->>,"h^{-1}","\cong"'] \\
    A \ar[r, >->,"f"] & B
\end{tikzcd},
\] the left diagram is a square in $\dcat{C}$ if and only if the right one is. 
\end{definition}

The purpose of \cref{defn:isostable} is to ensure the following lemma holds. 

\begin{lemma}\label{lem:isostable gives gpd}
    If $\dcat C$ is isostable, then $S^{\square}_\bullet \dcat C$ is a simplicial groupoid.
\end{lemma}\begin{proof}
If $w^v\dcat C$ consists of invertible morphisms, then every morphism in $S^{\square}_n\dcat C$ is pointwise invertible. Under the conditions of \cref{defn:isostable}, such a natural transformation admits an inverse natural transformation, whose components are the pointwise inverses.
\end{proof}

\begin{remark}\label{rmk:isostable}
It is not enough to assume that the vertical weak equivalences are all invertible to conclude that $S^{\square}_n\dcat C$ is a groupoid --- one further needs to ensure that squares are appropriately invertible. The subtlety is that, unlike in a 1-categorical setting, a vertical natural transformation of double functors which is pointwise an isomorphism may not admit an inverse natural transformation.
Indeed, suppose we have a vertical natural transformation $\tau\colon F\Rightarrow G$ between double functors $F,G\colon\dcat C\to\dcat D$ such that each component $\tau_C$ is a (vertical) isomorphism. In order for the maps $\tau_C^{-1}$ to assemble into a vertical natural transformation, we must now verify that for each horizontal map $f\colon C\cof C'$ in $\dcat C$ we have a square in $\dcat D$
    \[\begin{tikzcd}
     GC\ar[d, ->>, "\tau_C^{-1}"', "\cong"] \ar[r, >->,"Gf"]& GC' \ar[d, ->>,"\tau_{C'}^{-1}","\cong"'] \\
    FC \ar[r, >->,"Ff"] & FC'
\end{tikzcd}
\] which is not necessarily guaranteed unless we impose the conditions of \cref{defn:isostable}, which are relatively mild in practice.
\end{remark}

\begin{remark}\label{rmk:easywaytogetisostable}
    Suppose $\dcat C$ is a stable squares category whose weak equivalences are isomorphisms. A sufficient condition for $\dcat C$ to be isostable is to be able to choose span and cospan completion functors $s,t$ that behave a certain way when one of the morphisms in the (co)span is an equality:\[
    s\colon \begin{tikzcd}
        A \ar[r, >->, "f"] \ar[d, equal] & B\\
        A 
    \end{tikzcd} \longmapsto \begin{tikzcd}
        A \ar[r, >->, "f"] \ar[d, equal] \ar[rd, phantom, "\square"] & B\ar[d, equal] \\
        A \ar[r, >->, swap, "f"] & B
    \end{tikzcd}
\text{ and }
    s\colon \begin{tikzcd}
        A \ar[r, equal] \ar[d, ->>, swap, "g"] & A\\
        C 
    \end{tikzcd} \longmapsto \begin{tikzcd}
        A \ar[r, equal] \ar[d, ->>, swap, "g"] \ar[rd, phantom, "\square"] & A\ar[d, ->>, "g"] \\
        C \ar[r, equal] & C
    \end{tikzcd}
    \] and similarly for $t$. One can then deduce the final condition of isostability by first completing the span $A \xleftarrow{g^{-1}} C \xrightarrow{k} D$ and considering various compositions of squares.
\end{remark}

When $\dcat C$ is a proto-Waldhausen category whose weak equivalences are isomorphisms, we can show that the elements of $S_n \dcat C$ are completely determined by their top row. If $\dcat C$ is moreover stable, they are also determined by their rightmost column. We now prove these claims, and use them to obtain 2-Segal objects from stable squares categories.

\begin{definition}
    Let $\dcat C$ be a squares category. For each $n\geq 0$, we define $\cat H_n \dcat C$ as the category whose objects are length $n$ sequences of horizontal morphisms \[
O\cof A_1 \cof \dots \cof A_n 
\]
and whose morphisms are vertical natural transformations that are pointwise valued in $w^v\dcat C$. 
\end{definition}

\begin{definition} Similarly, define $\cat V_n\dcat C$ to be the category whose objects are length $n$ sequences of vertical morphisms\[
A_1 \quot \dots \quot A_n \quot O
\] and whose morphisms are vertical natural transformations that are pointwise valued in $w^v\dcat C$. 
\end{definition}

\begin{remark}
    One might hope that the categories $\cat H_n \dcat C$ assemble into a simplicial category, with faces and degeneracies defined in a similar manner as those of $S^{\square}_\bullet \dcat C$, using stability to define the face map $d_0$. However, this is not the case, as the simplicial identities do not hold. The analogous fact was remarked by Waldhausen \cite{waldhausen:1983} in the context of Waldhausen categories, and is the reason why the $S_\bullet$ construction must record the choices of quotients as part of the data. The same issue is present in $\cat V_n\dcat C$. We thank the anonymous referee for pointing out this subtlety.
\end{remark}

\begin{lemma}\label{lem:forget cokernels}
    Let $\dcat C$ be an isostable squares category.
   Then, for each $n\geq 0$, the forgetful functor $U_n\colon S^{\square}_n \dcat C\to \cat H_n \dcat C$ that takes an object to its top row is an equivalence of categories.
\end{lemma}
\begin{proof}
   We prove that each functor $U_n$ is an equivalence of categories by constructing an inverse equivalence. To do this, we will use the constructions (and notation) from \cref{prop:TplustoT} to define a section functor $F_n\colon\cat H_n\dcat C\to S^{\square}_n\dcat C$. Just as in \cref{prop:TplustoT}, the functor $F_n$ takes a sequence in $\cat H_n\dcat C$ to the staircase constructed by sequentially taking the span completions, which exist as $\dcat C$ is proto-Waldhausen. One can readily check that $F_n$ thus constructed takes maps in $\cat H_n\dcat C$ to maps in $S^{\square}_n\dcat C$; see for instance \cref{rmk:weakequivsforfree}. Clearly $U_n F_n=\id$, and we can construct a natural transformation $\tau\colon F_n U_n\Rightarrow\id$ exactly as in \cref{prop:TplustoT}; it suffices to show that $\tau$ is an isomorphism. 

   Now recall that each component $\tau_A\colon F_n U_n (A)\to A$ is a vertical natural transformation whose components are constructed inductively. For an arbitrary step, these are given by the composite
\[\begin{tikzcd}[bo column sep, bo row sep]
            &          & A_{i,j}\ar[rrr,>->]\ar[ddd,->>]   &     &       & A_{i,j+1}\ar[ddd,->>]\\
            &           &                               &            & \\
            & A_{i,j}\ar[rrr,>->,crossing over]\ar[ddd,->>] \ar[uur,equal] &                &                & A_{i,j+1}\ar[uur,equal] & \\
            &           & A_{i+1,j}\ar[rrr,>->]\ar[ddl,equal] &      &     & A_{i+1,j+1} \\
D_{i,j}\ar[rrr,>->,crossing over]\ar[ddd,->>] \ar[uur,->>]    &      &     & D_{i,j+1}\ar[uur,->>,crossing over] &           & \\
            & A_{i+1,j}\ar[rrr,>->] &    &       & X   \ar[uur,->>,gray] \ar[from=uuu,->>, crossing over]     & \\
             &           &            &                   &            & \\
D_{i+1,j}\ar[rrr,>->] \ar[uur,->>]&     &      &D_{i+1,j+1}\ar[uur,->>, gray]\ar[from=uuu,->>, crossing over]& & \\
\end{tikzcd}\]
Here, $D_{ij}=F_nU_n(A)_{ij}$, $X$ is the span completion of the middle face, and the map $X\twoheadrightarrow A_{i+1,j+1}$ is the component $w_{A_{i+1,j+1}}$ of the natural transformation given by condition (ii) of \cref{defn:protoWald cat}; hence, this map is a vertical isomorphism as all vertical weak equivalences in $\dcat C$ are isomorphisms by assumption. On the other hand, the map $D_{i+1,j+1}\twoheadrightarrow X$ is induced by the section $s$ given by condition (i) of \cref{defn:protoWald cat}. 

If we knew that the three given maps $D_{i,j}\twoheadrightarrow A_{i,j}$, $D_{i+1,j}\twoheadrightarrow A_{i+1,j}$ and $D_{i,j+1}\twoheadrightarrow A_{i,j+1}$ were vertical isomorphisms, then by \cref{rmk:isostable} these would give an invertible morphism of spans. Hence, this would ensure that the induced map $D_{i+1,j+1}\twoheadrightarrow X$ is also a vertical isomorphism (since $s$ preserves isomorphisms, by functoriality), and so the composite $D_{i+1,j+1}\twoheadrightarrow X\twoheadrightarrow A_{i+1,j+1}$ would be an isomorphism as well. This is indeed the case, 
as the corresponding maps are isomorphisms in the first inductive step (in fact, they are identities $A_{01}\to A_{01}$, $A_{02}\to A_{02}$ and $O\to O$).  

Thus $\tau_A\colon F_n U_n(A)\to A$ is a vertical natural transformation which is pointwise a vertical isomorphism; by \cref{rmk:isostable} this implies that each $\tau_A$ is invertible, and hence the natural transformation $\tau$ is itself invertible, as desired. 
\end{proof}

We now show the analogous lemma for $\cat V_n \dcat C$. One might hope that, by symmetry, the claim is immediate from \cref{lem:forget cokernels}. However, as highlighted in the following proof, the situation is more subtle; this lack of symmetry is discussed further in \cref{rmk:stablealternatives}.

\begin{lemma}\label{lem:forget kernels}
      Let $\dcat C$ be an isostable squares category. 
      Then, for each $n\geq 0$, the forgetful functor $U_n\colon S^{\square}_n \dcat C\to \cat V_n \dcat C$ that takes an object to its rightmost column is an equivalence of categories.
\end{lemma}\begin{proof}
    The proof is essentially the same as that of \cref{lem:forget cokernels}; the main difference lies in the fact that we must use the cospan completions given by the section $t$ in condition (iii) of \cref{defn:stable squares cat}, instead of span completions. 
    
    The functor $t$ allows us to construct a section functor $G_n\colon \cat V_n\dcat C\to S^{\square}_n\dcat C$ analogous to the functor $F_n$ in the proof of \cref{lem:forget cokernels}. The fact that $G_n$ takes maps in $\cat V_n\dcat C$ to maps in $S^{\square}_n\dcat C$ is less evident than its counterpart statement about $F_n$. To illustrate this, note for instance that given a map $f\colon A\to A'$ in $\cat V_n\dcat C$, the first inductive step in the definition of $G_n f$ uses the cospan completion functor $t$ to produce a cube
    \[\begin{tikzcd}[bo column sep, bo row sep]
  & Y'\ar[rr,>->]\ar[dd,->>] && A'_{n-2,n}\ar[dd,->>]\\
  Y\ar[ur,->>,gray,dashed]\ar[rr,>->,crossing over]\ar[dd,->>] && A_{n-2,n}\ar[ur,->>] &\\
  & O\ar[rr,>->] && A'_{n-1,n}\ar[dd,->>]\\
   O\ar[ur,equal]\ar[rr,>->] && A_{n-1,n}\ar[ur,->>]\ar[dd,->>] \ar[from=uu,->>, crossing over]&\\
   &&& O\\
&& O &\\
\end{tikzcd}\] However, even knowing that $A_{n-1,n}\twoheadrightarrow A'_{n-1,n}$ and $A_{n-2,n}\twoheadrightarrow A'_{n-2,n}$ are weak equivalences, general abstract nonsense is not enough to guarantee that the map $Y\twoheadrightarrow Y'$ is also a weak equivalence (in contrast to the span completion scenario). In this case, the fact that the morphism of the cospans is a pointwise vertical isomorphism implies it is an isomorphism of cospans (using \cref{rmk:isostable}) and consequently $t$ must send this morphism to an isomorphism of squares. Hence $Y\to Y'$ is also a vertical isomorphism. 

We still have that $U_n G_n=\id$, and we can construct a natural transformation $\eta\colon G_n U_n\Rightarrow\id$ analogous to the one defined in the proof of \cref{lem:forget cokernels}, using the natural transformation $u\colon tj^*\Rightarrow\id$ from condition (iv) of \cref{defn:stable squares cat}. Once again, we can show that $\eta$ is invertible, using the functoriality of $t$ and the fact that each vertical morphism $u_A$ is a weak equivalence, which is why this additional condition is required in \cref{defn:stable squares cat}.
\end{proof}

\begin{remark}\label{rmk:stablealternatives}
Our definition of stability was engineered for the two lemmas above to hold since, as we will soon see, they are instrumental in our strategy to obtain 2-Segal objects from squares categories. However, other avenues could be pursued to obtain these results as well: 
\begin{itemize}
    \item A detailed study of the proof of \cref{lem:forget kernels} reveals that the section functor $t$ is only applied to the subcategory $w^v\Fun^v\left(\tinycospan, \dcat C\right)$ of vertical natural transformations which are pointwise weak equivalences. With this in mind, one could modify condition (iii) in \cref{defn:stable squares cat} and ask for a section $t$ to the functor $j^*\colon w^v\Fun^v\left(\tinysquare, \dcat C\right)\to w^v\Fun^v \left(\tinycospan, \dcat C\right)$ instead. 
    \item A stability definition where conditions (i), (ii) are truly dual to (iii), (iv) would likely consider horizontal natural transformations between the cospan and square diagrams, instead of vertical ones. In this approach, it would not be necessary to add the requirement that the natural transformation $u$ is pointwise a weak equivalence, as this would be ensured by construction just as it is for $w$. The reason we choose not to do this is because we would not be able to prove \cref{lem:forget kernels} where each staircase in the $S^{\square}_\bullet$-construction is determined by its rightmost column, since our definition of $S^{\square}_\bullet$ has vertical natural transformations as its morphisms (see \cref{defn:S-dot of cat w squares}). Instead, following this approach would require us to define a horizontal version of $S^{\square}_\bullet$ as well, which would now be determined by its rightmost column whenever horizontal weak equivalences are isomorphisms, and then comparing the vertical and horizontal versions of $S^{\square}_\bullet$. These versions should agree as long as horizontal and vertical weak equivalences are in a bijective correspondence with each other in a way that is compatible with the squares in the double category. This holds, for instance, for any squares category arising from an ECGW-category \cite{SS-CGW}.
    \item The additional condition that $u$ is pointwise a weak equivalence would hold if the squares in $\dcat C$ satisfied the following property: in the picture below, whenever the diagram on the right and the outer diagram are squares in $\dcat C$, then so is the diagram on the left.
    \[\begin{tikzcd}
        A\rar[>->]\dar[->>] & B\rar[>->]\dar[->>] & C\dar[->>]\\
        A'\rar[>->] & B'\rar[>->] & C'
    \end{tikzcd}\] If so, we could do an argument similar to the one in \cref{rmk:weakequivsforfree}. This property is true, for instance, if the squares in $\dcat C$ are the cartesian squares in some ambient category. In this case, the assumption that $\dcat C$ is isostable is also automatic, as identity squares are always squares in any double category.
\end{itemize}
\end{remark}

We now follow the proof of \cite[Proposition 2.4.8]{dyckerhoff/kapranov:19} to show that the $S^{\square}_\bullet$-construction of an isostable squares category produces a $2$-Segal space. The idea is to use the comparison of $S^{\square}_n\dcat C$ with $\cat H_n \dcat C$ and $\cat V_n \dcat C$, and show the desired equivalences for these categories instead. The following two lemmas are the equivalences we will need.

\begin{lemma}\label{lem:lem for V}
    The map\[
   B\cat V_n(\dcat C)\to B\cat V_{\{0,\dots, n-2, n\}}(\dcat C) \times^{h}_{B\cat V_{\{n-2,n\}}(\dcat C)} B\cat V_{\{n-2, n-1,n\}}(\dcat C)
    \] is a homotopy equivalence for all $n\geq 3$.
\end{lemma}\begin{proof}
    We will show that the functors\[
   \psi_n\colon \cat V_n(\dcat C)\to \cat V_{\{0,\dots, n-2, n\}}(\dcat C) \times^{(2)}_{\cat V_{\{n-2,n\}}(\dcat C)} \cat V_{\{n-2, n-1,n\}}(\dcat C)
    \] are equivalences after realization for all $n\geq 3$. The target of $\psi_n$ is a projective $2$-limit of categories, described as follows (see \cite[Definition 1.3.6]{dyckerhoff/kapranov:19} for a general definition): \begin{itemize}
        \item An object is the data of objects \begin{align*}
            C_1\quot\dots\quot C_{n-2} \quot 0 &\in \cat V_{n-1}(\dcat C),\\
            D_{n-2}\quot D_{n-1}\quot 0 &\in \cat V_{2}(\dcat C)\\
            C\quot 0 \in \cat V_1(\dcat C)
        \end{align*}
        along with a span of vertical isomorphisms $\begin{tikzcd}
            C_{n-2} & C \ar[r, ->>, "\cong"] \ar[l, ->>, swap, "\cong"] & D_{n-2}.
        \end{tikzcd}$ 
        \item A morphism $(f_*, g_*, h)\colon (C_*, D_*, C)\to (C_*', D_*', C')$ is the data of morphisms $f_* \colon C_*\quot C_*'\in \cat V_{n-1}(\dcat C)$, $g_*\colon D_*\quot D_*'\in \cat V_{2}(\dcat C)$, and $h\colon C\to C'\in \cat V_1(\dcat C)$ so that the diagram\[
        \begin{tikzcd}
            C_{n-2} \ar[d, ->>, "f_j"] & C \ar[r, ->>, "\cong"] \ar[l, ->>, swap, "\cong"]\ar[d, ->>] & D_{n-2} \ar[d, ->>, "g_j"] \\
            C_{n-2}' & C' \ar[r, ->>, swap, "\cong"] \ar[l, ->>, "\cong"] & D_{n-2}'
        \end{tikzcd}
        \] commutes. 
    \end{itemize}
    The projective $2$-limit models the homotopy limit when the categories involved are groupoids (c.f. \cite[Proposition 1.3.8]{dyckerhoff/kapranov:19}) and hence it suffices to show $\psi_n$ induces an equivalence after geometric realization, as then\[
    B\cat V_n(\dcat C)\xrightarrow{\sim} B\left(\cat V_{\{0,\dots,n-2,n\}}(\dcat C) \times^{(2)}_{\cat V_{\{n-2,n\}}(\dcat C)} \cat V_{\{n-2,n-1,n\}}(\dcat C)\right) \xrightarrow{\sim} B\cat V_{n-1}(\dcat C) \times^{h}_{B\cat V_{1}(\dcat C)} B\cat V_{2}(\dcat C),
    \] since each $\cat V_k(\dcat C)$ is a groupoid. 

    We can represent an object $(C_*, D_*, C)$ as a commutative diagram\[
    \begin{tikzcd}
        && C_1 \ar[d, ->>] \\
        && \vdots \ar[d, ->>] \\
        D_{n-2} \ar[d, ->>] & C \ar[r, ->>,  "\cong"] \ar[l, ->>, swap, "\cong"] \ar[dd, ->>] & C_{n-2} \ar[dd, ->>] \\
        D_{n-1} \ar[d, ->>] &&\\
        0 & 0 \ar[r, equal] \ar[l, equal] & 0
    \end{tikzcd}.
    \] From such a diagram, we produce an element of $\cat V_n(\dcat C)$ as \[
    0 \quot C_1\quot \dots \quot C_{n-2}\quot D_{n-1} \quot O
    \] where the penultimate morphism is the composition $\begin{tikzcd}
        C_{n-2} \ar[r, ->>, "\cong"] & C \ar[r, ->>, "\cong"] & D_{n-2} \ar[r, ->>] & D_{n-1}
    \end{tikzcd}$. It is straightforward to check that this assignment extends to a functor \[q_n\colon \cat V_{\{0,\dots,n-2,n\}}(\dcat C) \times^{(2)}_{\cat V_{\{n-2,n\}}(\dcat C)} \cat V_{\{n-2,n-1,n\}}(\dcat C)\to \cat V_n(\dcat C)\] such that $q_n\circ \psi_n=\id$ and there is a natural transformation $\id\Rightarrow \psi_n\circ q_n$. The component of this natural transformation on an object $(C_*, D_*, C)$ is the morphism with $f_*=\id$, $g_*$ the identity everywhere except the first component where it is the composition $D_{n-2}\xrightarrow{\cong} C \xrightarrow{\cong} C_{n-2}$, and $h$ is given by the commutative diagram\[
        \begin{tikzcd}
            C_{n-2} \ar[d, equal] & C \ar[r, ->>, "\cong"] \ar[l, ->>, swap, "\cong"]\ar[d, ->>, "\cong"] & D_{n-2} \ar[d, ->>, "\cong"] \\
            C_{n-2} & C_{n-2} \ar[r, equal] \ar[l, equal] & C_{n-2}
        \end{tikzcd}.
        \] The existence of such a natural transformation implies $\psi_n$ is an equivalence after taking classifying spaces, and so we obtain the desired result.
\end{proof}

\begin{lemma}\label{lem:lem for H}
     The map\[
   B\cat H_n(\dcat C)\to B\cat H_{\{0,1,2\}}(\dcat C) \times^{h}_{B\cat H_{\{0,2\}}(\dcat C)} B\cat H_{\{0,2,\dots,n\}}(\dcat C)
    \] is an equivalence for all $n\geq 3$.
\end{lemma}\begin{proof}
    The proof for $\cat H_n(\dcat C)$ follows a similar idea as for $\cat V_n(\dcat C)$, but the arguments are more complicated as we need to accommodate the double categorical structure present in $\cat H_n(\dcat C)$. We are crucially making use of the fact that $\dcat C$ is isostable to ensure that each $\cat H_k (\dcat C)$ is a groupoid (see the argument in \cref{lem:isostable gives gpd} as well as \cref{rmk:isostable}). 
    
    It again suffices to show that the functors \[
   \phi_n\colon \cat H_n(\dcat C)\to \cat H_{\{0,1,2\}}(\dcat C) \times^{(2)}_{\cat H_{\{0,2\}}(\dcat C)} \cat H_{\{0,2,\dots, n\}}(\dcat C)
    \] are equivalences. In this case, an object of the target category is of the form\[
        \begin{tikzcd}
        0 \ar[r, >->] \ar[d, equal] \ar[rrd, phantom, "\square"] & C_1 \ar[r, >->] & C_2 & &\\
        0 \ar[d, equal] \ar[rr, >->] \ar[rrd, phantom, "\square"] & & C \ar[d, ->>, "\cong"]  \ar[u, ->>, swap, "\cong"] \\
        0 \ar[rr, >->] & & D_2 \ar[r, >->]& \dots\ar[r, >->]  & D_n
    \end{tikzcd}
        \] and a morphism $(f_*, g_*, h)\colon (C_*, C, D_*)\to (C_*', D_*', C')$ consists of $f_*\in \cat H_2(\dcat C)$, $g_*\in \cat H_{n-2}(\dcat C)$, and $h\in \cat H_1(\dcat C)$ such that \[
        \begin{tikzcd}
            C_2 \ar[d, ->>, "f_j"] & C \ar[r, ->>, "\cong"] \ar[l, ->>, swap, "\cong"]\ar[d, ->>, "h"] & D_2 \ar[d, ->>, "g_j"] \\
            C_2' & C' \ar[r, ->>, swap, "\cong"] \ar[l, ->>, "\cong"] & D_2'
        \end{tikzcd}
        \] commutes.
        
    To see that $\phi_n$ is an equivalence, we again construct a functor in the other direction which will be a homotopy inverse for $\phi_n$ after realization. Unlike in $\cat V_n(\dcat C)$, we cannot make use of \[\begin{tikzcd}
        C_2 \ar[r, ->>, "\cong"] & C \ar[r, ->>, "\cong"] & D_2 \ar[r, >->] & D_{3}
    \end{tikzcd}\]since we cannot compose the horizontal and vertical morphisms. However, given an element in $\cat H_{\{0,1,2\}}(\dcat C) \times^{(2)}_{\cat H_{\{0,2\}}(\dcat C)} \cat H_{\{0,2,\dots,n\}}(\dcat C)$, we obtain a diagram\[
    \begin{tikzcd}
        0 \ar[rr, >->]\ar[d, equal] \ar[rrd, phantom, "\square"] & & D_2 \ar[r, >->] \ar[d, ->>, "\cong"] & \dots\ar[r, >->]  & D_n\\
        0 \ar[d, equal] \ar[rr, >->] \ar[rrd, phantom, "\square"] & & C \ar[d, ->>, "\cong"]\\
        0 \ar[r, >->]  & C_1 \ar[r, >->] & C_2 & &
    \end{tikzcd}
    \] by inverting the given isomorphism in $\cat H_1(\dcat C)$ between $0\cof D_2$ and $0\cof C$. Now complete the bottom row to $0\cof C_1\cof C_2 \cof C_3'\cof \dots \cof C_n'$ where $C_i'$ is inductively constructed by completing the span\[
    \begin{tikzcd}
        D_{i-1}\ar[r, >->] \ar[d, ->>, swap, "\cong"] & D_{i}\\
        C_{i-1}' & 
    \end{tikzcd}
    \] for $3\leq i\leq n$, with $C_2'=C_2$. This assignment on objects extends to a functor $r_n\colon \cat H_{\{0,1,2\}}(\dcat C) \times^{(2)}_{\cat H_{\{0,2\}}(\dcat C)} \cat H_{\{0,2,\dots,n\}}(\dcat C)\to \cat H_n(\dcat C)$; the components of the natural transformation $r_n(f_*, g_*, h)$ which are not identities are produced by the span-completion functor $s$.

    We claim that there are vertical natural transformations $\id\Rightarrow \phi_n\circ r_n$ and $\id\Rightarrow r_n\circ \phi_n$, which completes the proof, as it implies that $r_n$ and $\phi_n$ are homotopy inverses after geometric realization. A component of the first natural transformation is given by the morphism\[
    \begin{tikzcd}
       && 0 \ar[rr, >->] \ar[dddd, gray, bend right, equal] && D_2 \ar[r, >->] \ar[dddd, gray, ->>] \ar[ddddr, phantom, "\square"] & D_3 \ar[r, >->] \ar[dddd, gray, ->>, "\cong"] \ar[ddddr, phantom, "\square"] & \cdots \ar[r, >->] \ar[dddd, phantom, "\dots"] \ar[ddddr, phantom, "\square"] & D_n \ar[dddd, gray, ->>, "\cong"] \\
       & 0 \ar[rr, >->,crossing over]\ar[dl, equal] \ar[ur, equal] \ar[dddd, gray, bend right, equal] && C \ar[dl, ->>, "\cong"] \ar[ur, ->>, "\cong"]  &&&&\\
       0\ar[r, >->,crossing over] \ar[dddd, gray, bend right, equal] &C_1\ar[r, >->,crossing over]  & C_2 &&&&&\\
       &&&&&&&\\
       && 0 \ar[rr, >->]  && C_2 \ar[r, >->] & C_3'  \ar[r, >->] & \cdots \ar[r, >->] & C_n'\\
       & 0 \ar[rr, >->]\ar[dl, equal] \ar[ur, equal] && C_2 \ar[dl, equal] \ar[ur, equal] \ar[from=uuuu, gray, bend right=10, ->>, "\cong", crossing over]&&&&\\
       0\ar[r, >->]  &C_1\ar[r, >->]\ar[from=uuuu, gray, bend right, equal,crossing over] & C_2 \ar[from=uuuu, gray, bend right, equal,crossing over]&&&&&\\
    \end{tikzcd}
    \] and naturality is again ensured by the span-completing functor $s$. 

    For $\id\Rightarrow r_n\circ \phi_n$, a component is given by\[
    \begin{tikzcd}
        0 \ar[r, >->] \ar[d, gray, equal] & C_1\ar[r, >->] \ar[d, gray, equal] & C_2 \ar[r, >->] \ar[d, gray, equal]& C_3\ar[r, >->] \ar[d, gray, ->>, "\cong"] & \cdots \ar[r, >->] \ar[d, phantom, "\cdots"] & C_n\ar[d, gray, ->>, "\cong"]\\
        0\ar[r, >->] & C_1 \ar[r, >->]& C_2 \ar[r, >->]& C_3' \ar[r, >->]& \cdots \ar[r, >->]& C_n'
    \end{tikzcd}
    \] where $C_i'$ is obtained inductively by completing the span\[
    \begin{tikzcd}
        C_{i-1} \ar[r, >->] \ar[d, ->>] & C_i \\
        C_{i-1}'  & 
    \end{tikzcd}
    \] (with $C_2'=C_2$). It is straightforward to check that these components assemble into a natural transformation, again using properties of $s$.
\end{proof}

\begin{theorem}\label{thm:2Segal from squares}
    If $\dcat C$ is an isostable squares category, then $[n]\mapsto BS_n(\cat C)$ is a $2$-Segal space.
\end{theorem}
\begin{proof}
We will show that the functors\begin{align*}
           \Phi_n\colon S^{\square}_n(\dcat C)&\to S^{\square}_{\{0,1,2\}}(\dcat C) \times^{(2)}_{S^{\square}_{\{0,2\}}(\dcat C)} S^{\square}_{\{0,2,\dots,n\}}(\dcat C)\\
           \Psi_n\colon S^{\square}_n(\dcat C)&\to S^{\square}_{\{0,\dots,n-2,n\}}(\dcat C) \times^{(2)}_{S^{\square}_{\{n-2,n\}}(\dcat C)} S^{\square}_{\{n-2,n-1,n\}}(\dcat C)
    \end{align*}
are equivalences after realization for all $n\geq 3$. Then, by the criterion given in \cref{rmk:check 2Segal}, the same argument as in the previous two lemmas shows that $[n]\mapsto BS^{\square}_n \dcat C$ is a $2$-Segal space, since $S^{\square}_\bullet(\dcat C)$ is a simplicial groupoid (\cref{lem:isostable gives gpd}). 

For $\Phi_n$, consider the commutative diagram\[
    \begin{tikzcd}
        S^{\square}_n(\dcat C) \ar[r, "\Phi_n"] \ar[d] & S^{\square}_{\{0,1,2\}}(\dcat C) \times^{(2)}_{S^{\square}_{\{0,2\}}(\dcat C)} S^{\square}_{\{0,2,\dots,n\}}(\dcat C) \ar[d] \\
        \cat H_n\dcat C \ar[r, swap, "\phi_n"] & \cat H_{\{0,2,1\}}\dcat C \times^{(2)}_{\cat H_{\{0,2\}}(\dcat C)} \cat H_{\{0,2,\dots, n\}}\dcat C
    \end{tikzcd}
    \] where the vertical maps are induced by the forgetful functor. By \cref{lem:forget cokernels} and \cref{lem:lem for H}, the vertical functors and bottom horizontal functors are equivalences after realization, and hence the top horizontal functor is as well.

    Similarly, for $\Psi_n$, consider the commutative diagram\[
    \begin{tikzcd}
        S^{\square}_n(\dcat C) \ar[r, "\Psi_n"] \ar[d] & S^{\square}_{\{0,\dots,n-2,n\}}(\dcat C) \times^{(2)}_{S^{\square}_{\{n-2,n\}}(\dcat C)} S^{\square}_{\{0,\dots,n-2\}}(\dcat C) \ar[d] \\
        \cat V_n\dcat C \ar[r, swap, "\phi_n"] & \cat V_{\{0,\dots, n-2, n\}}\dcat C \times^{(2)}_{\cat V_{\{n-2,n\}}(\dcat C)} \cat V_{\{n-2,n-1,n\}}\dcat C
    \end{tikzcd}
    \] where the vertical maps are induced by the forgetful functor, and apply \cref{lem:forget kernels} and \cref{lem:lem for V}.
\end{proof}

\begin{remark}
In particular, the above result (together with \cref{thm:square and S comparison}) tells us that there is a model for the algebraic $K$-theory of a stable squares category with isomorphisms which produces a 2-Segal space.
Note that in order to say this, a passage from $T_\bullet$ to $S^{\square}_\bullet$ (and the corresponding comparison) is truly required. Indeed, the realization of the $T_\bullet$-construction does not generally produce a 2-Segal space even when all weak equivalences are isomorphisms, as $T_\bullet(\dcat C)$ is not a simplicial groupoid. For this to happen, \emph{all vertical morphisms} in $\dcat C$ must be isomorphisms, which is not the case in any examples of interest. 
\end{remark}

\begin{example}
    The examples of finite sets (\cref{ex:FinSet}) and polytopes (\cref{ex:polytopes}) are isostable. Hence, by \cref{thm:2Segal from squares}, their $S^{\square}_\bullet$-constructions yield $2$-Segal spaces.
\end{example}

   To conclude  this paper, we summarize the comparison between some of the different double-categorical structures which were introduced in the current paper or previously present in the literature. First note that every stable pointed double category (introduced in \cite[Section 3]{BOORS:2016}) is an example of an isostable squares category (given in \cref{defn:isostable}). This claim follows from \cref{prop:stable dblcat is stable sqcat}, together with the observation that the weak equivalences in a stable pointed double category are identities (hence invertible), and that the square condition in \cref{defn:isostable} is guaranteed by \cref{rmk:easywaytogetisostable}, as by definition, each span in a stable pointed double category has a unique completion to a square.  Since every isostable squares category is a stable one (by definition), we have a chain of inclusions\[
    \boxed{\parbox{8em}{\centering pointed stable\\double categories}}  ~\subsetneq~  \boxed{\parbox{8em}{\centering isostable squares\\ categories}} ~\subsetneq~  \boxed{\parbox{8em}{\centering stable squares\\ categories}}.
   \] 
   Note that these are all proper inclusions. Indeed, any stable squares category whose weak equivalences are not isomorphisms cannot be isostable; and any isostable squares category in which squares are not \emph{uniquely} determined by their span and cospan cannot be pointed stable double categories, as is the case of the isostable squares categories of finite sets and of polytopes of \cref{ex:FinSet,ex:polytopes} whose (co)spans only determine the squares \emph{up to isomorphism}.

   The conditions required of an isostable squares category can then be interpreted as a more flexible version of the ones in a pointed stable double category, where now squares intuitively behave like pushouts and pullbacks, and well-behaved ``up to isomorphism'' constructions are allowed. On the one hand, this added flexibility has the consequence that we cannot claim to produce 2-Segal \emph{sets} from the $K$-theory of isostable squares categories; this is not surprising at all, as it was already understood in \cite{BOORS:2016} that pointed stable double categories are precisely the data required to produce 2-Segal sets. On the other hand, thanks to \cref{thm:2Segal from squares}, these double categories still fit nicely into this story, giving us a wider class of double-categorical inputs that produce 2-Segal \emph{spaces}.

\printbibliography
\end{document}